\documentclass{amsart}

\usepackage[all]{xy}
\usepackage{hyperref}
\usepackage{amssymb}
\usepackage{mathdots}

\usepackage{tikz,tkz-euclide}
\usepackage{tikz-3dplot}

\tdplotsetmaincoords{70}{165}

\usetkzobj{all}

\hypersetup{
    colorlinks,    citecolor=blue,    filecolor=blue,    linkcolor=blue,    urlcolor=blue
}

\usepackage{booktabs}
\usepackage{mathrsfs}
\usepackage{todonotes}
\usepackage{bbm}
\usepackage[lite,bibtex-style]{amsrefs}
\usepackage{longtable}

\usepackage{MnSymbol}

\newtheorem{introthm}{Theorem}

\newtheorem{introprop}[introthm]{Proposition}
\newtheorem{theorem}{Theorem}[section]
\newtheorem{lemma}[theorem]{Lemma}
\newtheorem{proposition}[theorem]{Proposition}
\newtheorem{corollary}[theorem]{Corollary}

\theoremstyle{definition}
\newtheorem{notation}[theorem]{Notation}
\newtheorem{definition}[theorem]{Definition}
\newtheorem{example}[theorem]{Example}

\newtheorem{remark}[theorem]{Remark}
\theoremstyle{remark}

\numberwithin{equation}{section}

\def\cc{{\mathbb C}}

\def\zz{{\mathbb Z}}
\def\rr{{\mathbb R}}

\def\qq{{\mathbb Q}}
\def\pp{{\mathbb P}}

\newcommand{\GL}{\operatorname{GL}}
\newcommand{\Lt}{\operatorname{Lt}}
\newcommand{\PNef}{\operatorname{PNef}}
\newcommand{\NE}{\operatorname{NE}}
\newcommand{\Nef}{\operatorname{Nef}}
\newcommand{\Hom}{\operatorname{Hom}}
\newcommand{\lw}{\operatorname{lw}}
\newcommand{\cone}{\operatorname{cone}}
\newcommand{\rker}{\operatorname{rker}}

\newcommand{\im}{\operatorname{im}}
\newcommand{\lcm}{\operatorname{lcm}}

\newcommand{\codim}{\operatorname{codim}}

\begin{document}

\title{On base loci of higher fundamental 
forms of toric varieties}

\author[A.~Laface]{Antonio Laface}
\address{
Departamento de Matem\'atica,
Universidad de Concepci\'on,
Casilla 160-C,
Concepci\'on, Chile}
\email{alaface@udec.cl}

\author[L.~Ugaglia]{Luca Ugaglia}
\address{
Dipartimento di Matematica e Informatica,
Universit\`a degli studi di Palermo,
Via Archirafi 34,
90123 Palermo, Italy}
\email{luca.ugaglia@unipa.it}

\thanks{
The first author was partially supported 
by Proyecto FONDECYT Regular N. 119077.
The second author is member of INdAM - GNSAGA.
Both authors have been partially supported 
by project Anillo ACT 1415 PIA Conicyt.
}

\begin{abstract}
We study the base locus
of the higher fundamental forms of
a projective toric variety $X$ at a general
point. More precisely we consider 
the closure $X$ of the image of a
map $(\cc^*)^k\to \pp^n$, 
sending $t$ to the vector of Laurent 
monomials with exponents $p_0,\dots,p_n\in \zz^k$.
We prove that the $m$-th
fundamental form of such an $X$ at a general point
has non empty base locus if and only if the points
$p_i$ lie on a suitable degree-$m$ affine hypersurface.

We then restrict to the case in which the 
points $p_i$ are all the lattice points of a 
lattice polytope and we give some applications
of the above result. In particular we provide a 
classification for the second fundamental forms
on toric surfaces, and we also 
give some new examples of weighted 
$3$-dimensional projective spaces whose 
blowing up at a general point is not Mori dream.

\end{abstract}

\maketitle

\section*{Introduction}
Let $X\subseteq\mathbb P^n$
be a projective variety and let
$q\in X$ be a general point.
Denote by $\pi\,\colon\tilde X\to X$
the blowing-up of $X$ at $q$ with
exceptional divisor $E$. Given a
hyperplane section $H$ of $X$
it is an open problem to provide
necessary and sufficient conditions
on the embedding $X\to\mathbb P^n$
in order for the linear 
system $|\pi^*H-mE|$ to be {\em special},
which means that its dimension is bigger 
than the expected one.
The problem has been widely studied 
in case $m=2$, see for instance
~\cites{AH,CGG,CC} and the references therein,
but it remains open even in this case.
For higher values of $m$ there are
conjectures when $X$ is the blowing 
up of $\mathbb P^2$ (Segre-Harbourne-Gimigliano-Hirschowitz Conjecture~\cites{G,H,Hi,Se})
and $\mathbb P^3$ (Laface-Ugaglia 
Conjecture~\cites{LU,LU1})
at points in very general position.
These conjectures predict that 
a necessary condition 
for $|\pi^*H-mE|$ to be special is 
that it has positive dimensional base locus.

In this paper we investigate the above 
problem in case $X$ is the closure
of a monomial embedding 
$(\mathbb C^*)^k\to\mathbb P^n$,
so that $X$ is a not necessarily normal
toric variety. The principal tool
that we use is 
the restricted linear system
\[
 |\pi^*H-mE|_E,
\]
which is also called the {\em $m$-th 
fundamental form} of $X$ at $q$
(see for instance~\cites{GH,IR}).
The $m$-th fundamental form turns
out to be useful in two directions. On one
hand, a base point for the $m$-th fundamental
form is a base point for the system $|\pi^*H-mE|$
too. On the other hand, the dimension of 
the $m$-th form turns out to be related to the
speciality of the system $|\pi^*H-(m+1)E|$
(see Proposition~\ref{equiv}).

In order to state our results, let us
fix a $k$-dimensional lattice
$M\simeq \zz^k$ and a finite set of
points $S=\{p_0,\dots,p_n\}\subseteq M$.
It is possible to define a map 
$f\colon (\cc^*)^k\to \pp^n$
which associates to $t$ the vector of Laurent 
monomials with exponents $p_0,\dots,p_n$.
The closure of the image of the above map is
a $k$-dimensional projective toric variety 
$X(S)\subseteq \pp^n$, and we denote by 
$\mathbf 1\in X(S)$ the image of the neutral 
element of $(\cc^*)^k$. The point $\mathbf 1$ 
lies in the open torus orbit, and hence it is 
a general point of $X(S)$. 
An element $v\in N := \Hom(M,\mathbb Z)$
defines a map $\mathbb C^*\to X$ by
$t\mapsto f(t^v)$, whose derivative 
at $t=1$ is a vector of $T_{\mathbf 1}X$.
This induces a linear map 
$N\otimes_{\mathbb Z}\mathbb C\to
T_{\mathbf 1}X$ which allows us to identify 
$\mathbb P(N\otimes_{\mathbb Z}\mathbb C)$
with $\mathbb P(T_{\mathbf 1}X)\simeq
\mathbb P^{k-1}$ in our main theorem.
In~\cite{Pe} it is shown that the $m$-th fundamental
form at $\mathbf 1$ is not the complete linear system
if and only if the points $p_0,\dots,p_n$ lie on
an affine hypersurface of degree $m$.
Our main result shows that the $m$-th fundamental form at $\mathbf 1$ 
has a base point if and only if the top degree part of the 
above affine hypersurface is a pure power. More precisely we
have the following (see also Example~\ref{ex:par} and~\ref{ex:pal}
for the difference between our result and Perkinson's).

\begin{introthm}
\label{bs}
Given an integer $m\geq 2$
the following are equivalent:
\begin{enumerate}
\item
the $m$-th fundamental form 
at $\mathbf 1\in X(S)$ has a base
point $[v]\in\mathbb P(T_{\mathbf 1}X)$;
\item
the points of $S$ lie on an affine 
hypersurface of 
$M\otimes_\mathbb Z\mathbb C$ 
of equation
\[
 (v\cdot x)^m+ \text{lower degree terms}= 0.
\]
\end{enumerate}
\end{introthm}
We then restrict to the case of a toric variety
associated to a polytope. 
Indeed, given a full-dimensional lattice polytope
$\Delta\subseteq M\otimes_{\mathbb Z}\qq$, it is 
possible to define a polarized pair $(X,H)$, where 
$X=X(\Delta)$ is the projective toric variety
associated to the lattice points $\Delta\cap M$, 
while $H$ is a very ample divisor of $X$.  
In what follows we will denote by $\pi\, \colon \tilde X\to X$ 
the blowing up of the toric variety $X$ along the 
point $\mathbf 1$ and by $E$ the exceptional 
divisor.
A first consequence of Theorem~\ref{bs} is the
following characterisation of projective toric surfaces 
whose second fundamental form at $\mathbf 1$ is
not full dimensional (see also Definition~\ref{def:lw}):
\begin{introprop}
\label{pro:pol}
Let $\Delta\subseteq M\otimes_{\mathbb Z}\qq \simeq \qq^2$ 
be a full dimensional lattice polytope such that 
$|\Delta\cap M|\geq 6$ and let $(X,H)$ be the 
corresponding polarized pair.
Then the following are equivalent:
\begin{enumerate}
\item 
the second fundamental
form of $X$ at $\mathbf 1$ 
is not full dimensional;
\item
the linear system $|\pi^*H-3E|$ is special;
\item
$\Delta$ is either a Cayley polygon 
or it is equivalent, modulo $\GL(2,\zz)$, 
to one of the following:
\begin{center}
\footnotesize
\begin{longtable}{cccc}
Type & & Vertices 
&
\\
\hline
\\
$(i)$ 
&
$
 \begin{tikzpicture}[scale=.4]
 \tkzDefPoint(2,-1){P1}
 \tkzDefPoint(0,0){P2}
 \tkzDefPoint(-3,-1){P3}
 \tkzDefPoint(0,-2){P4}
 \tkzFillPolygon[color = black!20](P1,P2,P3,P4)
 \tkzDrawSegments[color=black](P1,P2 P2,P3 P3,P4 P4,P1)
 \foreach \x/\y in {0/-1,1/-1,0/0,0/-2,2/-1,-1/-1,-3/-1, -2/-1}
 {
  \tkzDefPoint(\x,\y){p}
  \tkzDrawPoints[fill=black,color=black,size=5](p)
 }
\end{tikzpicture}
$
&
$(a,0),\, (0,1),\, (-b,0),\, (0,-1)$,
& 
with $a \geq 1,\, b \geq 0$ and $a+b \geq 3$
\\
&&&
\\
\hline
\\
$(ii)$ 
&
$
 \begin{tikzpicture}[scale=.4]
 \tkzDefPoint(2,0){P1}
 \tkzDefPoint(1,1){P2}
 \tkzDefPoint(-3,0){P3}
 \tkzDefPoint(0,-1){P4}
 \tkzFillPolygon[color = black!20](P1,P2,P3,P4)
 \tkzDrawSegments[color=black](P1,P2 P2,P3 P3,P4 P4,P1)
 \foreach \x/\y in {0/0,1/0,1/1,0/-1,-1/0, -2/0, -3/0, 2/0}
 {
  \tkzDefPoint(\x,\y){p}
  \tkzDrawPoints[fill=black,color=black,size=5](p)
 }
\end{tikzpicture}
$
&
$(a,0),\, (0,1),\, (-b,0),\, (-1,-1)$,
& 
with $a \geq 1,\, b \geq 0$ and $a+b \geq 3$
\\
&&&
\\
\hline
\end{longtable}
\end{center}
\end{enumerate}
In particular the second fundamental form at $\mathbf 1$ 
has non empty base locus if and only if $\Delta$ is Cayley.
\end{introprop}
Going back to the problem stated 
at the beginning of the introduction,
an easy corollary of the above result is
that if the linear system $|\pi^*H-3E|$
is special, then its base locus contains a curve
(the strict transform of the closure 
of a one-parameter subgroup) intersecting
$E$.
We will show that if $m \geq 4$, this is no longer
true, i.e. there are examples 
of special linear systems of the form $|\pi^*H-mE|$ 
whose base locus does not contain such a curve
(see Example~\ref{ex:base}).

Finally, when $k \geq 2$, we make use 
of Theorem~\ref{bs}  in order to study stable 
base loci of divisors of the form $\pi^*H-mE$
on $\tilde X$.
In particular we give a sufficient condition
on $\Delta$ implying that $\pi^*H-mE$ 
is not semiample (Corollary~\ref{cor:notsem})
and as an application we provide the following 
new list of $3$-dimensional weighted projective spaces 
$\pp(a_1,\dots,a_4)$, with $a_i\leq 30$, 
whose blowing up at $\mathbf 1$ is not a
Mori dream space.

\begin{introprop}
\label{propo:nonmds}
Let $X := \mathbb P(a_1,\dots,a_4)$ 
and let $H$ be an ample divisor of degree 
${\rm lcm}(a_1,\dots,a_4)$.
If the vector of weights is in the following 
table then the divisor $\pi^* H-mE$ is 
nef but not semiample. In particular 
the blowing up of $X$ at
$\mathbf 1$ is not Mori dream.
{\footnotesize
\begin{table}[h]
\begin{tabular}{|c|c|}
\toprule
$[a_1,\dots,a_4]$ & $m$ \\
\midrule
$[ 7, 11, 13, 15 ]$ & $572$\\
$[ 7, 13, 16, 19 ]$ & $832$\\
$[ 7, 15, 19, 23 ]$ & $1140$\\
$[ 7, 17, 22, 27 ]$ & $1496$\\
$[ 7, 19, 20, 24 ]$ & $380$\\
$[ 7, 23, 25, 29 ]$ & $2300$\\
$[ 9, 10, 13, 17 ]$ & $702$\\
$[ 9, 13, 16, 23 ]$ & $1152$\\
$[ 9, 16, 19, 20 ]$ & $342$\\
$[ 9, 16, 19, 29 ]$ & $1710$\\
$[ 9, 17, 23, 28 ]$ & $2070$\\
$[ 9, 19, 22, 26 ]$ & $990$\\
$[ 9, 25, 28, 29 ]$ & $3024$\\
$[ 10, 11, 16, 19 ]$ & $480$\\
$[ 10, 11, 17, 23 ]$ & $1122$\\
$[ 10, 13, 17, 18 ]$ & $540$\\
$[ 10, 13, 21, 29 ]$ & $1638$\\
$[ 10, 17, 19, 21 ]$ & $1520$\\
$[ 10, 17, 22, 23 ]$ & $880$\\
$[ 10, 17, 24, 29 ]$ & $986$\\
$[ 10, 19, 21, 24 ]$ & $320$\\
$[ 10, 19, 23, 27 ]$ & $2300$\\
$[ 10, 19, 23, 28 ]$ & $1064$\\
$[ 10, 21, 22, 27 ]$ & $378$\\
$[ 10, 21, 23, 26 ]$ & $1196$\\
$[ 10, 21, 26, 29 ]$ & $1300$\\
$[ 10, 23, 27, 28 ]$ & $1400$\\
$[ 11, 12, 13, 17 ]$ & $816$\\
$[ 11, 13, 23, 28 ]$ & $1794$\\
$[ 11, 15, 19, 24 ]$ & $480$\\
$[ 11, 16, 25, 28 ]$ & $550$\\
\bottomrule
\end{tabular}\hspace{5mm}
\begin{tabular}{|c|c|}
\toprule
$[a_1,\dots,a_4]$ & $m$ \\
\midrule
$[ 11, 17, 25, 29 ]$ & $2550$\\
$[ 11, 18, 20, 21 ]$ & $280$\\
$[ 11, 19, 24, 26 ]$ & $1248$\\
$[ 11, 20, 21, 27 ]$ & $756$\\
$[ 11, 23, 24, 28 ]$ & $644$\\
$[ 11, 23, 25, 28 ]$ & $2800$\\
$[ 12, 13, 16, 19 ]$ & $304$\\
$[ 12, 13, 17, 22 ]$ & $663$\\
$[ 12, 17, 19, 23 ]$ & $1656$\\
$[ 12, 17, 19, 25 ]$ & $1800$\\
$[ 12, 17, 20, 23 ]$ & $460$\\
$[ 12, 17, 25, 26 ]$ & $1275$\\
$[ 12, 19, 22, 25 ]$ & $1100$\\
$[ 12, 19, 25, 28 ]$ & $700$\\
$[ 12, 23, 25, 29 ]$ & $2784$\\
$[ 12, 23, 26, 29 ]$ & $1508$\\
$[ 13, 14, 15, 22 ]$ & $616$\\
$[ 13, 14, 17, 25 ]$ & $1638$\\
$[ 13, 15, 17, 27 ]$ & $540$\\
$[ 13, 15, 24, 29 ]$ & $754$\\
$[ 13, 16, 19, 27 ]$ & $1728$\\
$[ 13, 17, 23, 29 ]$ & $2392$\\
$[ 13, 17, 24, 25 ]$ & $2550$\\
$[ 13, 18, 22, 29 ]$ & $1276$\\
$[ 13, 19, 21, 29 ]$ & $2436$\\
$[ 13, 20, 21, 29 ]$ & $2520$\\
$[ 13, 21, 28, 30 ]$ & $80$\\
$[ 14, 17, 22, 27 ]$ & $1232$\\
$[ 14, 17, 23, 24 ]$ & $1224$\\
$[ 14, 17, 24, 29 ]$ & $1392$\\
$[ 14, 19, 23, 30 ]$ & $1260$\\
\bottomrule
\end{tabular}\hspace{5mm}
\begin{tabular}{|c|c|}
\toprule
$[a_1,\dots,a_4]$ & $m$ \\
\midrule
$[ 14, 19, 27, 29 ]$ & $3192$\\
$[ 16, 17, 19, 22 ]$ & $969$\\
$[ 16, 18, 19, 29 ]$ & $1296$\\
$[ 16, 19, 20, 29 ]$ & $696$\\
$[ 16, 21, 23, 26 ]$ & $1449$\\
$[ 16, 22, 25, 27 ]$ & $1782$\\
$[ 17, 18, 20, 27 ]$ & $162$\\
$[ 17, 20, 21, 23 ]$ & $2520$\\
$[ 17, 20, 26, 27 ]$ & $1620$\\
$[ 17, 21, 22, 23 ]$ & $2772$\\
$[ 17, 21, 22, 29 ]$ & $3213$\\
$[ 17, 21, 24, 29 ]$ & $1218$\\
$[ 17, 23, 25, 26 ]$ & $3519$\\
$[ 17, 23, 25, 29 ]$ & $3450$\\
$[ 17, 23, 26, 30 ]$ & $2070$\\
$[ 17, 23, 27, 29 ]$ & $3726$\\
$[ 17, 25, 27, 29 ]$ & $4050$\\
$[ 18, 19, 21, 28 ]$ & $76$\\
$[ 18, 20, 23, 27 ]$ & $189$\\
$[ 18, 23, 26, 27 ]$ & $216$\\
$[ 18, 26, 27, 29 ]$ & $243$\\
$[ 19, 20, 22, 29 ]$ & $1740$\\
$[ 19, 22, 24, 25 ]$ & $1584$\\
$[ 19, 22, 25, 26 ]$ & $1672$\\
$[ 19, 23, 24, 25 ]$ & $3312$\\
$[ 19, 24, 27, 29 ]$ & $1296$\\
$[ 19, 24, 29, 30 ]$ & $696$\\
$[ 19, 25, 26, 27 ]$ & $3952$\\
$[ 19, 25, 28, 29 ]$ & $4275$\\
$[ 22, 25, 27, 28 ]$ & $2025$\\
$[ 23, 27, 29, 30 ]$ & $1827$\\
\bottomrule
\end{tabular}
\label{tab}
\end{table}
}
\end{introprop}

In~\cite{GKK1} and~\cite{He} there are examples of 
$3$-dimensional weighted projective spaces whose 
blowing up at $\mathbf 1$ is not Mori dream. 
We remark that there is no intersection between our 
list and the one of~\cite{He},
since we consider only the cases in which 
no weight $a_i$ belongs to the semigroup
generated by the remaining ones.
Concerning the list of~\cite{GKK1}, there is only one 
common case, namely $\pp(17,18,20,27)$
(see also Remark~\ref{rem:ex}).

The paper is structured as follows. 
In Section~\ref{pre} we first introduce 
higher fundamental forms on projective
varieties, then we recall some definitions 
and facts about projective toric varieties
and finally we specialize to the case
of toric varieties associated to a lattice polytope
$\Delta$. 
In Section~\ref{proof:bs}
we prove Theorem~\ref{bs} and we 
present a couple of related examples.
The last section deals with some applications
of Theorem~\ref{bs} to toric varieties associated
to a polytope $\Delta$.
In particular we first prove a corollary which gives a 
condition on $\Delta$ implying
that a suitable divisor on the blowing up of the toric
variety $X(\Delta)$ is not semiample. Then
we consider the dimension $2$ case, proving 
Proposition~\ref{pro:pol} and 
some related results, and finally we restrict to
weighted projective spaces, proving
Proposition~\ref{propo:nonmds}.

\subsection*{Acknowledgements}
We thank the anonymous referee whose
comments helped us to improve 
the exposition of the paper.

\section{Preliminaries}
In this section we begin by recalling the
definition of the $m$-th fundamental form of a 
projective variety, and then the definition
of the projective toric variety $X\subseteq\mathbb P^n$
associated to a set $S = \{p_0,\dots,p_n\}$ of lattice points.
Finally we restrict to the case 
in which $S$ is the set of all lattice points in a lattice
polytope $\Delta$. 

\label{pre}
\subsection{Fundamental forms}
\label{ff}
We recall the following definition (see ~\cite{IR}*{Definition 1.1}).
Let $X\subseteq\mathbb P^n$ be a projective
variety of dimension $k$, and let $H$ be a hyperplane 
section of $X$. Given a point $q\in X$ denote by
$\pi\,\colon \tilde X\to X$ the blowing-up of
$X$ at $q$, and by $E$ the exceptional 
divisor.
The {\em $m$-th fundamental form} of $X$ at $q$
is the linear system of degree $m$ homogeneous
polynomials of $\mathbb P^{k-1}$ defined by the
image of the restriction map
\begin{equation}
\label{rho}
  \rho_m\,\colon
  H^0(\tilde X,\mathcal O(\pi^*H-mE))
  \to
  H^0(E,\mathcal O(\pi^*H-mE)).
\end{equation}
Recall also that $X\subseteq
\mathbb P^n$ is 
{\em linearly normal} if
the complete linear system
$|H|$ has dimension $n$.
For example a smooth rational quartic curve of $\mathbb P^n$ 
is linearly normal only if $n=4$.

\begin{remark}
\label{rem:bs}
As we wrote in the introduction, the $m$-th fundamental form of $X$
at $q$ is related to the two divisors $\pi^*H -mE$
and $\pi^*H - (m+1)E$. Indeed, from the definition 
it follows immediately that if the $m$-th form
has a base point, then the divisor $\pi^*H -mE$ on 
$\tilde X$ has a base point too. Moreover, concerning
$\pi^*H - (m+1)E$ we have the following.
\end{remark}

\begin{proposition}
\label{equiv}
Assume that $X\subseteq\mathbb P^n$ 
is linearly normal and that
the $i$-th fundamental form of $X$ at $q$ 
is full dimensional, for $2\leq i\leq m-1$. 
Then the following are equivalent:
\begin{enumerate}
\item
the $m$-th fundamental form of $X$ at 
$q$ is not full dimensional,
i.e. $\rho_m$ is not surjective;
\item
the linear system $|\pi^*H-(m+1)E|$ does not have
the expected dimension
$n-\binom{m+2}{2}$.
\end{enumerate}
\end{proposition}
\begin{proof}
Denote by $h_i$ the dimension of the 
vector space $H^0(\tilde X,\mathcal O(\pi^*H-iE))$.
By hypothesis 
we have $h_{i+1} = h_i - \binom{i+k-1}{k-1}$
for any $0\leq i\leq m-1$,
so that $h_m = h_0 -
\sum_{i=0}^{m-1}\binom{i+k-1}{k-1} = n + 1 -
\binom{m-1+k}{k}$, where the last equality
is due to the linear normality of $X$ and to an
elementary property of binomial coefficients.
It follows that
\[
 h_{m+1} = n + 1 - \binom{m-1+k}{k} - \dim(\im(\rho_m))
 = n + 1 - \binom{m+k}{k} + \codim(\im(\rho_m)),
\]
which proves the statement.
\end{proof}

We are interested in describing
the fundamental forms of a 
variety $X$ in parametric form
in the following sense.
Let $U$ be an open subset 
of $\mathbb C^k$ and let
\begin{equation}
\label{param}
 f\colon U\to\mathbb P^n
 \qquad
 (u_1,\dots,u_k)\mapsto [f_0 : \dots : f_{n}]
\end{equation}
be a morphism whose image has dimension 
$k$ as well. Denote by $X$ the Zariski closure
of the image
and take a smooth point $p\in U$ 
such that its image $q = f(p)\in X$ is smooth 
as well.
\begin{remark}
\label{first-iso}
In this setting we have the following isomorphism
\[
 H^0(\tilde X,\mathcal O(\pi^*H-mE))
 \to
 H^0(X,\mathcal O(H)
 \otimes\mathcal I(q)^m)
 \qquad
 \sigma\mapsto \pi_*(\sigma),
\]
whose inverse is $x\mapsto
\pi^*(x)\sigma_E^{-m}$,
where $\sigma_E\in H^0(\tilde X,E)$ is a non-zero section.
The pullback of the sheaf 
$\mathcal O_X(H)\otimes
\mathcal I(q)^m$
via $f$ is isomorphic to 
$\mathcal O_U\otimes \mathcal I(p)^m$. Since $U$ is quasiaffine the 
latter sheaf is generated by 
its global sections. Thus
we have the isomorphism
\[
 H^0(X,\mathcal O(H)\otimes
\mathcal I(q)^m)
 \to
 \langle f_0,\dots,f_n\rangle
 \cap I(p)^m
 \qquad
 x\mapsto f^*(x),
\]
where we denote by $\langle f_0,\dots,f_n\rangle$ the 
complex vector space generated by the $f_i$. 
As a consequence the domain 
of $\rho_m$ is isomorphic to 
the subspace 
of $\langle f_0,\dots,f_{n}\rangle$
consisting of elements whose partial
derivatives of order up to $m-1$ vanish 
at $p$. 
\end{remark}
Given a multi-index 
$\alpha = (\alpha_1,\dots,\alpha_k)
\in \mathbb Z^k_{\geq 0}$ let us denote by 
$\partial^\alpha g|_p$ the partial derivative 
of the function $g$, defined by $\alpha$ and evaluated 
at $p\in U$. We recall the following definition 
(see also~\cites{DP,Pe}).
\begin{definition}
\label{mjets}
The {\em matrix of $m$-jets} of $f$ at $p\in U$, 
denoted by $J_m(f)|_p$, is the vertical join 
of the matrices
\[
 D_r(f)|_p
 := 
 \left(\partial^\alpha f_i|_p\, \colon\, |\alpha| = r 
 \text{ and } 0\leq i\leq n\right),
\]
for $r=0,\dots,m$.
\end{definition}

\subsection{Toric varieties}
 \label{subs:tor}
We recall here some basic facts about 
projective toric varieties (see for 
instance~\cites{CLS,Fu,St}),
and we introduce the notion of pseudonef cones for
the blowing up of a toric variety at a general point.

In what follows $M$ will be a rank $k$ free abelian 
group, $N := {\rm Hom}(M,\mathbb Z)$ its dual,
and $S = \{p_0,\dots,p_n\}\subseteq M$ 
a finite subset whose differences generate $M$.
For any $p\in M$ denote by 
$\chi^p(u)\in\mathbb C[u_1^{\pm 1},\dots,u_k^{\pm 1}]$ 
the corresponding Laurent monomial.
The closure $X(S)$ of the image of the morphism 
\begin{equation}
 \label{eq:imm}
 f\colon (\mathbb C^*)^k
 \to\mathbb P^n
 \qquad
 u\mapsto [\chi^{p_0}(u)\ :\ \dots :\chi^{p_n}(u)],
\end{equation}
is the {\em projective toric variety} defined by $S$,
and we denote by $\mathbf 1\in X(S)$
the image of the neutral element 
$1\in (\mathbb C^*)^k$. 
The toric variety $X:=X(S)$ defined
in this way is in general non normal.
In what follows we will denote by
$\pi\,\colon \tilde X\to X$  the blowing up
of $X$ at the point $\mathbf 1$ (which lies 
in the open torus orbit) and by
$E$ the exceptional divisor of $\pi$ or,
with abuse of notation, its class.

\begin{notation}
\label{def:param}
Given an element $v\in N$, in what follows
we denote by $C_v\subseteq X$ the Zariski closure
of the image of the map 
$t\mapsto f(t^v)$,
and by $\tilde C_v\subseteq \tilde X$ 
its strict transform.
We denote by $A_i(X)$ (resp. $A_i(X)_{\mathbb Q}$)
the $i-th$ (resp. rational) Chow group of $X$. 
Let ${NE}(X)\subseteq A_1(X)_{\mathbb Q}$
be the Mori cone of $X$, let $e\in {\rm NE}(\tilde X)$
the class of a line contained
in the exceptional divisor $E$ and 
let $\Gamma_N\subseteq{\rm NE}(\tilde X)$
be the subset consisting of the classes of the curves 
$\tilde C_v\subseteq\tilde X$,
where $v\in N$.
Finally we denote by 
$\pi_*\colon A_i(\tilde X)\to A_i(X)$ the homomorphism
induced by pushforward of cycles~\cite{fu2}*{\S 1.4}.
\end{notation}

\begin{definition}
 \label{def:psnef}
The {\em pseudonef cone} of $\tilde X$ is
\[
 \PNef(\tilde X)
 :=
 {\cone}(\pi_*^{-1}(\NE(X))\cap E^{\perp},e,\Gamma_N)^*,
\] 
and a {\em pseudonef class} is a class in ${\PNef}(\tilde X)$.
\end{definition}
\begin{proposition}
We have the inclusion $\Nef(\tilde X) \subseteq \PNef(\tilde X)$.
\end{proposition}
\begin{proof}
It is enough to prove that 
the cone $\pi_*^{-1}(\NE(X))\cap E^{\perp}$
is contained in the Mori cone $\NE(\tilde X)$. 
In order to do that, observe that given a nef 
class $\tilde D\in \Nef(\tilde X)$,
we can write $\tilde D = \pi^*D-mE$, where
$D\in\Nef(X)$ and $m\geq 0$. 
Therefore, for any $\gamma$ in
$\pi_*^{-1}(\NE(X))\cap E^{\perp}$
we have
\[
\gamma\cdot \tilde D = \gamma\cdot \pi^*D
= \pi_*\gamma \cdot D \geq 0,
\]
where the first equality follows from 
the fact that $\gamma$ lies in $E^\perp$,
and the second from
the projection formula (see~\cite{laz}*{pag. 17}).
We conclude that $\gamma$ has
non negative intersection product 
with any nef class of $\tilde X$, and
the claim follows. 
\end{proof}

We remark that in general the other inclusion does not
hold, since there are cases in which a pseudonef
divisor is not nef, nor even effective, as we
are going to see in Example~\ref{ex:noteff}.
But in Proposition~\ref{nef:p} we are going
to prove that in some interesting cases we do have
the equality $\PNef(\tilde X) = \Nef(\tilde X)$.

\subsection{Toric varieties associated to a 
lattice polytope}
 \label{subs:pol}
Most of the paper will deal with the case in which 
$S = \Delta \cap M$ 
is the set  of all the lattice points of a full-dimensional
lattice polytope $\Delta\subseteq M\otimes_{\mathbb Z}\qq$,
and hence we recall here some definitions and observations.

Given a lattice polytope $\Delta\subseteq M\otimes_{\mathbb Z}\qq$,
let us consider the set $\Delta \cap M$ of its lattice points
and let us denote simply by $X(\Delta)$ the toric variety
$X(\Delta \cap M)$.
The polytope $\Delta$ defines indeed a polarized pair $(X,H)$
consisting of the projective toric variety
$X := X(\Delta)$, together with a very ample
divisor $H$ of $X$ (and in particular $X$
is linearly normal).
Let us recall the following definitions 
(see for instance~\cites{AWW,LS}).

\begin{definition}
\label{def:lw}
Given a {\em lattice direction}, i.e. a non zero primitive vector 
$v\in N$, let us denote respectively by $\min\langle\Delta,v\rangle$ and 
$\max\langle\Delta,v\rangle$ the minimum and the maximum of 
$\langle m,v\rangle$ for $m\in \Delta$. The
{\em lattice width of} $\Delta$ {\em in the direction}
$v$ can be defined as
\[
 \lw_v(\Delta) := \max\langle\Delta,v\rangle  - 
  \min\langle\Delta,v\rangle.
\]
The {\em lattice width of} $\Delta$ is defined as 
\[
 \lw(\Delta) := \min\{\lw_v(\Delta)\, :\,  v\in N\},
\]
and if $v\in N$ is such that $\lw_v(\Delta) = \lw(\Delta)$,
we say that $v$ is a {\em width direction} for $\Delta$.
The polytope $\Delta$ is 
called a {\em Cayley polytope}
if $\lw(\Delta) = 1$.
\end{definition}
\begin{remark}
\label{lwidth}
Given a lattice polytope $\Delta$,
if $v\in N$ is a width direction, then $\Delta$ 
is bounded by the two hyperplanes 
$L_{\min}(\Delta,v):=\{  \langle x,v\rangle = \min\langle\Delta,v\rangle\}$
and $L_{\max}(\Delta,v):=\{ \langle x,v\rangle = \max\langle\Delta,v\rangle\}$
respectively. Therefore, all the lattice points of $\Delta$
lie on the union of $\lw(\Delta)+1$ hyperplanes orthogonal to $v$.
In particular, $\Delta$ is a Cayley polytope iff all its lattice
points lie on two parallel hyperplanes at lattice distance $1$.
\end{remark}

\begin{remark}
\label{rem:deg}
Let $(X,H)$ be the toric polarized pair defined
by the lattice polytope $\Delta$. Then $\lw_v(\Delta)$ 
is the degree of the curve $C_v$
(see Notation~\ref{def:param}).
\end{remark}

\begin{remark}
\label{rem:spec}
Let $\Delta$ and $(X,H)$ be as before and 
let us consider an integer $m\geq 2$.
Proposition~\ref{equiv} implies that 
$m$ is the smallest degree of an affine 
hypersurface passing through all the lattice 
points of $\Delta$ if and only if
$|\pi^*H-iE|$ has the expected dimension for 
$0\leq i\leq m$, while
$|\pi^*H-(m+1)E|$ does not have the expected dimension.
This is essentially the content of~\cite{Pe}*{Proposition~1.1}, 
since $|\pi^*H-(m+1)E|$ corresponds to hyperplane 
sections of $X$ containing the $(m+1)$-th osculating
space to $X$ at $\mathbf 1$.
\end{remark}

\begin{proposition}
\label{prop:psnef}
Let $\Delta$ and $(X,H)$ be as above.
Then the following are equivalent.
\begin{enumerate}
\item
$\pi^*H - mE$ is pseudonef;
\item
$0\leq m\leq \lw(\Delta)$.
\end{enumerate}
\end{proposition}
\begin{proof}
We prove $(1)\Rightarrow (2)$.
First of all, since $E \cdot e = -1$, 
we have that $(\pi^*H - mE)\cdot e = m \geq 0$. 
Therefore we have $(\pi^*H -mE)\cdot
\tilde{C}_v \geq 0$, for any $\tilde{C}_v$
in $\Gamma_N$ (see Notation~\ref{def:param}).
Thus by Remark~\ref{rem:deg} we deduce 
that $\lw_v(\Delta) - m \geq 0$ for any $v\in N$ and,
by taking $v$ to be a width direction
we obtain the second inequality.

We prove $(2)\Rightarrow (1)$.
By hypothesis $H$ is very ample, so that
$\pi^*H$ is nef and thus $\pi^*H - mE$ 
has non-negative intersection with 
any class in 
$E^\perp\cap{\rm NE}(\tilde X)$. 
Moreover, by the same arguments
given above, it follows that $\pi^*H - mE$
has non-negative intersection with $e$
and all the curves $\tilde{C}_v$
in $\Gamma_N$.
\end{proof}

The proposition above gives a characterisation 
of divisors $\pi^*H-mE$, for $0\leq m\leq \lw(\Delta)$.
If $m$ is bigger than $\lw(\Delta)$ we do not
have such a characterisation, but we can prove
that some implications hold. This is the content
of the following:

\begin{proposition}
\label{prop:w}
Let $(X,H)$ be the toric polarized pair defined
by the lattice polytope $\Delta$, and let $v\in N$
be a width direction for $\Delta$. 
Given an integer $m\geq 2$ and the following statements:
\begin{enumerate}
\item
\label{c1}
$m \geq \lw(\Delta) + 1$;
\item
\label{c2}
the curve $\tilde C_v$ 
is contained in the base locus of  $|\pi^*H - mE|$;
\item
\label{c3}
the $m$-th fundamental form of
$X$ at $\mathbf 1$ has a base point
at $[v] \in \pp(T_{X,\mathbf 1})$;
\item
\label{c4}
$h^1(\pi^*H - (m+1)E)  > 0$;
\end{enumerate}
the implications~
\eqref{c1}$
\Rightarrow
$\eqref{c2}$
\Rightarrow
$\eqref{c3}$
\Rightarrow
$\eqref{c4} hold.
\end{proposition}
\begin{proof}
Let us suppose that~\eqref{c1} holds. If we denote 
as before by $\tilde C_v$ the strict transform
of the rational curve $C_v$,
we have that $(\pi^*H-mE)\cdot
\tilde C_v = \lw(\Delta) - m \leq -1$, 
and~\eqref{c2} follows.

If ~\eqref{c2} holds, then
the tangent direction $[v]$ to the base curve at $\mathbf 1$
is a base point for the $m$-th fundamental
form at $\mathbf 1$.

Finally, if~\eqref{c3} holds then the restriction
map $\rho_m$ in the following exact sequence 
\[
 \xymatrix@C=12pt{
  0\ar[r] &
  H^0(\tilde X,\pi^*H-(m+1)E)\ar[r] &
  H^0(\tilde X,\pi^*H-mE)\ar[r]^-{\rho_m} &
  H^0(E,\pi^*H-mE_{|E})
 }
\]
is not surjective, which immediately implies 
~\eqref{c4}.
\end{proof}

In the next section we will see (Remark~\ref{rem:teo})
that there are some cases in which the conditions above
are indeed equivalent, but in general only the given
implications hold .

\vspace{5mm}

We conclude this section with an example
of a pseudonef divisor which is not effective.
\begin{example}
 \label{ex:noteff}
Let $\Delta$ be the lattice triangle of
vertices $(90,0),(0,117),
(0,130)$ so that $X$ is the weighted projective plane 
$\mathbb P(9,10,13)$, and $H = 1170A$, 
were $A$ is the ample generator of the 
divisor class group of $X$.
The one parameter subgroup
of smallest degree of $X$ 
is the one defined by the 
equation $x^4-yz^2=0$ and it has degree $36$.
Thus the pseudonef cone is 
generated by the classe
of the exceptional divisor $E$ and the class of 
$D := \pi^*65A-2E$, which has intersection product $0$
with the strict transform of the above one parameter subgroup.
To prove that $D$ lies outside the effective cone
it suffices to show that it
has negative intersection 
product with some nef class $C$. 
Such a $C$ is the strict transform of the 
irreducible curve defined by an element of 
a basis of the $4$-th saturated power 
of the lattice ideal~\eqref{eq:I}
(see~\cite{hkl}*{\S 5}).
More explicitly
\[
 \begin{array}{r}
2x^{14}z - 3x^{11}y^4 - 5x^{10}yz^3 + 9x^7y^5z^2 + 3x^6y^2z^5 + x^4y^9z\\[7pt]
    - 12x^3y^6z^4 + 4x^2y^3z^7 - xy^{13} - xz^{10} + 3y^{10}z^3=0.
 \end{array}
\]
To check that the curve has multiplicity $4$ at the point
$\mathbf{1} \in X$ observe that the invariant characters 
in Cox coordinates are generated by 
$u := \frac{x^4}{yz^2}$ and 
$v := \frac{z^3}{xy^3}$.
Dividing the above polynomial
by $y^{10}z^3$ we get a degree
zero polynomial which can thus
be written as a Laurent polynomial
in the torus coordinates $u,v$.
After performing the coordinate
change and multiplying by $v$
we get
\[
 2u^4v^3 - 5u^3v^3 - 3u^3v^2 + 3u^2v^3 + 9u^2v^2 - uv^4 + 4uv^3 - 
    12uv^2 + uv + 3v - 1 = 0.
\]
Translating $(1,1)$ to the origin 
one sees that the multiplicity is $4$.
By the above discussion we have
$C = \pi^*139A-4E$, so that $C^2 = \frac{139^2}{1170}-16 > 0$.
Then $C$ is nef, being the class of an irreducible 
curve with positive self-intersection,
and we conclude observing that 
$C\cdot D = \frac{65\cdot 139}{1170}-8 < 0$.
In the following picture 
we show the position of 
$D$ and $C$ with respect 
to the light cone of 
$\tilde X$, which is the 
grey region.

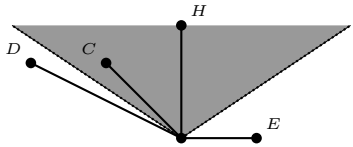
\begin{center}
\begin{tikzpicture}
\draw[line width=0mm,fill=gray!80!white] (2.25,1.5) -- (0,0) -- (-2.25,1.5);
\foreach \x/\y in {0/1.5,1/0,-1/1,-2/1}
 {\draw[-, thick] (0,0) -- (\x,\y); }
\foreach \x/\y in {0/0,0/1.5,1/0,-1/1,-2/1}
 { \fill[black] (\x,\y) circle (2pt); }
 \draw[-, thick,densely dotted] (0,0) -- (-2.25,1.5);
 \draw[-, thick,densely dotted] (0,0) -- (2.25,1.5);
\node[above left] at (-2,1) {\tiny $D$};
\node[above left] at (-1,1) {\tiny $C$};
\node[above right] at (1,0) {\tiny $E$};
\node[above right] at (0,1.5) {\tiny $H$};
\end{tikzpicture}
\end{center}
\end{example}

\section{Proof of Theorem~\ref{bs}}
\label{proof:bs}

In this section we are going to prove 
Theorem~\ref{bs}, by following the 
idea of Perkinson~\cite{Pe}, i.e. 
the study of suitable relations between 
left and right kernels of a slight modification
of the $m$-th jet matrix.

Let us consider as in Subsection~\ref{subs:tor}
any finite set $S = \{p_0,\dots,p_n\}\subseteq M$ 
of lattice points,
and the map $f\colon (\cc^*)^k \to \pp^n$
given by the Laurent monomials 
$\chi^p$, for $p\in S$, whose image is 
the toric variety $X=X(S)$ (see~\eqref{eq:imm}).
Given an integer $m\geq 2$, 
for simplicity of notation we will set
\[
 J_m := J_m(f)|_1
 \quad
{\rm and}
\quad
 D_r := D_r(f)|_1,
\]
for any $0\leq r\leq m$ (see Definition~\ref{mjets}).
The columns of the matrices $J_m$ and 
$D_r$ are indexed by the points $p_0,\dots,p_n$,
while the rows of $D_r$ (resp. $J_m$) 
are indexed by the partial derivatives 
$\partial^\alpha$ of order $|\alpha| = r$ 
(resp. $0\leq |\alpha| \leq m$) in $k$
variables. We fix the graded lexicographical 
order on these derivatives.
Given $\alpha=(\alpha_1,\dots,\alpha_k)$
we define the following 
polynomials of $\mathbb C[x_1,\dots,x_k]$
\[
 P_\alpha
 :=
 \prod_{i=1}^kx_i(x_i-1)\cdots (x_i-\alpha_i+1)
 \quad
 {\rm and}
 \quad
 \Lt(P_\alpha)
 :=
 \prod_{i=1}^kx_i^{\alpha_i},
\]
where the product $x_i(x_i-1)\cdots (x_i-\alpha_i+1)$
is to be intended $1$ if $\alpha_i = 0$.
Observe that the $(i+1)$-th column of the matrix 
$J_m$ is $P_\alpha(p_i)$, for $0\leq |\alpha|\leq 
m$. 
If we denote by $\Lt(J_m)$ and $\Lt(D_m)$ 
the matrices of leading terms of $J_m$ 
and $D_m$ respectively,  then by the definitions above
the $(i+1)$-th column of $\Lt(J_m)$ consists of 
$\Lt(P_\alpha)(p_i)$, for $0\leq |\alpha|\leq m$ 
(see also~\cites{Ca,DP,Pe}).

\begin{remark}
\label{lt}
The matrix $\Lt(J_m)$ is obtained from $J_m$ by
elementary row operations. In particular
the two matrices have the same row span,
the same rank,
the same right kernel, and left kernels
of the same dimension.
An element $c = (c_p\, :\, p\in S)$ in the right kernel 
of $J_m$ corresponds to the Laurent polynomial
\[
 R_c(u) := \sum_{p\in S}c_p\chi^p(u),
\]
whose derivatives at $1\in (\mathbb C^*)^k$
vanish up to order $m$. In other words,
the zero locus of $R_c$ is a hyperplane section
of $X$ which has multiplicity at least
$m+1$ at $\mathbf 1\in X$.
On the other hand, since the $(i+1)$-th 
column of $\Lt (J_m)$ consists 
of all the monomials of degree up 
to $m$ of $\mathbb C[x_1,\dots,x_k]$
evaluated at $p_i$, an element 
$b = (b_\alpha\, :\, 0\leq |\alpha|\leq m)$
in the left kernel of $\Lt (J_m)$ corresponds 
to the polynomial
\[
 L_b 
 := 
 \sum_{0\leq |\alpha|\leq m}b_\alpha\Lt (P_\alpha)
\]
of $\mathbb C[x_1,\dots,x_k]$, of degree
at most $m$, which vanishes at
all the points of $S$.
\end{remark}

\begin{example}
\label{ex:parab}
If we consider $S=\{(0,0),(1,0),(0,1),(3,1),(1,3),(6,3)\}\subseteq \zz^2$,
and we fix $m=2$, we have

\begin{align*}
J_2 & = 
\left(\phantom{
\begin{matrix} 1\\ 2 \\ 3 \\ 4 \\ 5 \\ 6 
\end{matrix}}
\right.
\hspace{-.7em}
\begin{matrix}
1 & 1 & 1 & 1 & 1 & 1\\
\hline
0 & 1 & 0 & 3 & 1 & 6\\
0 & 0 & 1 & 1 & 3 & 3\\
\hline
0 & 0 & 0 & 6 & 0 & 30\\
0 & 0 & 0 & 3 & 3 & 18\\
0 & 0 & 0 & 0 & 6 & 6
\end{matrix}
\hspace{-.6em}
\left.\phantom{\begin{matrix}1 \\ 2 \\ 3 \\ 4 \\ 5 \\ 6 \end{matrix}}\right)
\hspace{-.8em}
\begin{tabular}{l}
$\left.\lefteqn{\phantom{\begin{matrix} 1 \end{matrix}}}\right\} D_0$\\
$\left.\lefteqn{\phantom{\begin{matrix} 2 \\ 3 \end{matrix}}}\right\} D_1$\\
$\left.\lefteqn{\phantom{\begin{matrix} 4 \\ 5 \\ 6 \end{matrix}}} \right\} D_2$
\end{tabular}
&
\hspace{5mm}
\Lt(J_2) & = 
\left(\phantom{
\begin{matrix} 1\\ 2 \\ 3 \\ 4 \\ 5 \\ 6 
\end{matrix}}
\right.
\hspace{-.7em}
\begin{matrix}
1 & 1 & 1 & 1 & 1 & 1\\
\hline
0 & 1 & 0 & 3 & 1 & 6\\
0 & 0 & 1 & 1 & 3 & 3\\
\hline
0 & 1 & 0 & 9 & 1 & 36\\
0 & 0 & 0 & 3 & 3 & 18\\
0 & 0 & 1 & 1 & 9 & 9
\end{matrix}
\hspace{-.6em}
\left.\phantom{\begin{matrix}1 \\ 2 \\ 3 \\ 4 \\ 5 \\ 6 \end{matrix}}\right)
\hspace{-.8em}
\begin{tabular}{l}
$\left.\lefteqn{\phantom{\begin{matrix} 1 \end{matrix}}}\right\} \Lt(D_0)$\\
$\left.\lefteqn{\phantom{\begin{matrix} 2 \\ 3 \end{matrix}}}\right\} \Lt(D_1)$\\
$\left.\lefteqn{\phantom{\begin{matrix} 4 \\ 5 \\ 6 \end{matrix}}} \right\} \Lt(D_2)$
\end{tabular}.
\end{align*}

\vspace{3mm}
\noindent The right kernel of $J_2$ (which
coincides with the right kernel of
$\Lt(J_2)$) is generated by the
vector $c=( 10, -10,  -5,   5,   1,  -1)$, 
corresponding to the polynomial
$R_c = 10-10u_1-5u_2+5u_1^3u_2+u_1u_2^3-u_1^6u_2^3$,
vanishing at $(1,1)$ together with its derivatives 
up to order $2$. 

On the other hand, the left kernel of $\Lt(J_2)$
is generated by $b = (0,-1,-1,1,-2,1)$, corresponding
to the degree $2$ polynomial 
$L_b = -x_1-x_2+x_1^2-2x_1x_2+x_2^2$,
vanishing at all the points of $S$.

\end{example}

\begin{proof}[Proof of Theorem~\ref{bs}]
Let $\mathfrak m_\mathbf 1$ be the ideal sheaf
of the point $\mathbf 1\in X$, and let 
$T_{X,\mathbf 1}$ and $T^*_{X,\mathbf 1}$ be 
the tangent and cotangent spaces 
of $X$ at $\mathbf 1$ respectively.
We fix $w_1,\dots,w_k$ to be 
coordinates on $T_{X,\mathbf 1}$, and
we denote by ${\rker}$ 
the right kernel of a matrix. 
We then have a commutative diagram: 
\begin{equation}
\label{secd}
 \xymatrix@C=90pt{
   H^0(\tilde X,\pi^*H-mE)\ar[r]^-{\rho_m} \ar[d]^-\simeq
   & H^0(E,\pi^*H-mE)\ar[d]^-\simeq\\
  {\rker}(J_{m-1})\ar[r]^-{
 g\mapsto
  \sum_{|\alpha|=m}w^\alpha
  \frac{m!}{\alpha_1!\cdots\alpha_k!}\,
\partial^{\alpha}g|_1
}  
  & {\rm Sym}^m(T^*_{X,\mathbf 1})
 }
\end{equation}
where the left vertical arrow 
is the isomorphism described 
in Remark~\ref{first-iso}, while the 
right vertical isomorphism follows from 
the usual identification
$\mathcal O_E(-E)\simeq 
\mathfrak m_\mathbf 1/\mathfrak m_\mathbf 1^2$
and the fact that the latter 
quotient is isomorphic to $T^*_{X,\mathbf 1}$.
Let us fix the following notation:
\[
 \Phi_m(w) 
 :=
 \left(w^\alpha \frac{m!}{\alpha_1!\cdots\alpha_k!}\, :\, |\alpha| = m\right)
 \cdot D_m
 \in\mathbb C[w_1,\dots,w_k]^{\binom{m+k}{k}},
\]
so that $\Phi_m(w)$ is a vector of 
homogeneous polynomials of 
degree $m$ in $k$ variables,
with $\binom{m+k}{k}$ entries.
By~\eqref{secd}, the $m$-th fundamental form of 
$X$ at $\mathbf 1$ corresponds to the linear 
system in $\mathbb P^{k-1}$ defined 
by the following degree $m$
homogeneous polynomials of 
$\mathbb C[w_1,\dots,w_k]$
\[
 \sum_{|\alpha|=m}w^\alpha
 \,\frac{m!}{\alpha_1!\cdots\alpha_k!}\,
\partial^\alpha \Big(\sum_{p\in S}c_p\chi^{p}(u)\Big)\\
  =
 \Phi_m(w)\cdot c,
\]
as  $c$ varies in ${\rker}(J_{m-1})$.
Therefore the $m$-th fundamental form has a base
point at $[v]\in\mathbb P^{k-1}$ if and only if
$\Phi_m(v)\cdot c = 0$, for any $c\in{\rker}(J_{m-1})$
or equivalently if
\begin{equation}
\label{eq:phi}
\Phi_m(v) 
  \in {\rker}(J_{m-1})^\perp
  = 
(\text{row span}(J_{m-1})^\perp)^\perp
  = 
 \text{row span}(J_{m-1}).
\end{equation}
If we define $\Lt (\Phi_m)(w)$
as the vector of polynomials
obtained by replacing the matrix $D_m$
with $\Lt (D_m)$ in the definition
of $\Phi_m(w)$, we have that 
$ \Lt (\Phi_m)(v) - \Phi_m(v)$ 
belongs to the row span of $J_{m-1}$. 
By~\eqref{eq:phi} and Remark~\ref{lt} we deduce that 
\[
 \Lt (\Phi_m)(v) \in \text{row span}(\Lt(J_{m-1})),
\]
which gives a non trivial vector in the left kernel of 
$\Lt (J_m)$ (recall that $\Lt(J_m)$ is the vertical
join of $\Lt(J_{m-1})$ and $\Lt(D_m)$).
By the same Remark~\ref{lt}, this vector corresponds to a 
hypersurface of degree $m$ containing 
the points $p \in S$. We conclude by
observing that we can write $\Lt (\Phi_m)(v)
=
((v\cdot p_i)^m \, :\, 0\leq i\leq n)$, so that 
the defining polynomial for the hypersurface
takes the form
\[
 (v\cdot x)^m + \text{lower degree terms},
\]
which gives the statement.
\end{proof}

\begin{example}
 \label{ex:par}
Let us consider again the 
set $S\subseteq \zz^2$ of 
Example~\ref{ex:parab}. The polynomial 
$L_b$ we found before can be written as
$(x_1-x_2)^2-x_1-x_2 
=
(v\cdot x)^2 - x_1 - x_2$, where $v=(1,-1)$. 
This is the equation of a parabola
passing through the points of $S$.
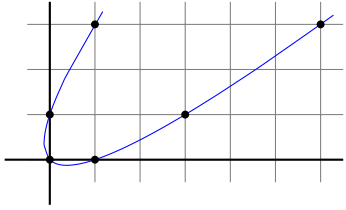
\vspace{5mm}

\begin{center}
\begin{tikzpicture}[x=.6cm,y=.6cm]
 \draw[step=1.0,gray,thin] (-0.5,-0.5) grid (6.5,3.5);
 \draw[thick] (0,-1) -- (0,3.5);
 \draw[thick] (-1,0) -- (6.5,0);
 \draw[blue, smooth, domain=-1:.51] plot ({(\x+1)/(\x-1)^2}, {\x*(\x+1)/(\x-1)^2});
\draw[blue, smooth, domain=2.8:50] plot ({(\x+1)/(\x-1)^2}, {(\x^2+\x)/(\x-1)^2});
\draw[blue, smooth, domain=-35:-1.1] plot ({(\x+1)/(\x-1)^2}, {\x*(\x+1)/(\x-1)^2});

\foreach \x/\y in {0/0,0/1,1/0, 3/1, 1/3, 6/3}
 { \fill[black] (\x,\y) circle (1.5pt); }
\end{tikzpicture}
\end{center}
\vspace{5mm}
Indeed, the second fundamental form of the 
surface $X(S)\subseteq\pp^5$ at
$\mathbf 1$ is  the $1$ dimensional linear
system generated by $w_1(w_1+w_2)$ 
and $w_2(w_1+w_2)$, so that it has 
the base point $[v]=(1:-1)\in\pp^1$.

\end{example}

We are now going to present another example
of toric surface in order to stress the difference between
Perkinson's result~\cite{Pe} and Theorem~\ref{bs}.
This is the well known Togliatti surface 
(see also~\cite{To} and~\cite{Pe}*{Example~2.4}).

\begin{example}
 \label{ex:pal}
If we now consider $S=\{(0,0),(1,0),(0,1),(2,1),(1,2),(2,2)\}\subseteq \zz^2$,
and $m=2$, we have that the left kernel of $\Lt(J_2)$ is generated
by $(0,1,1,-1,1,-1)$, corresponding to the conic 
$V(x_1+x_2-x_1^2+x_1x_2-x_2^2)$ 
containing all the points of $S$.

 
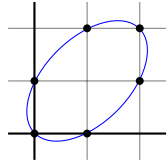
\vspace{5mm}

\begin{center}
\begin{tikzpicture}[x=.7cm,y=.7cm]
 \draw[step=1.0,gray,thin] (-0.5,-0.5) grid (2.5,2.5);
 \draw[thick] (0,-.5) -- (0,2.5);
 \draw[thick] (-.5,0) -- (2.5,0);

\draw[rotate around={-45:(1,1)},blue,smooth] (1,1) ellipse (16pt and 28pt);

\foreach \x/\y in {0/0,0/1,1/0, 2/1, 1/2, 2/2}
 { \fill[black] (\x,\y) circle (1.5pt); }
\end{tikzpicture}
\end{center}
\vspace{5mm}
In this case the second fundamental form
is the $1$ dimensional linear system generated by $w_1^2 - w_2^2$ and
$2w_1w_2 + w_2^2$. Indeed, according to 
Perkinson's result, the second form is not full dimensional
because the points of $S$ lie on a conic, but since the conic
is not a parabola, according to Theorem~\ref{bs} the 
form has no base point.
 
 \end{example}

\section{Applications}

From now on we are going to apply the
above results to projective toric varieties $X:=X(\Delta)$, 
associated to a full-dimensional lattice polytope 
$\Delta\subseteq M\otimes_{\mathbb Z}\qq$
(see Subsection~\ref{subs:pol}).

In particular, we first prove
a corollary of Theorem~\ref{bs},
which allows to construct some interesting non semiample
divisors on $\tilde X$, then we restrict to the 
case of toric surfaces and finally we consider weighted 
projective spaces.

\begin{remark}
 \label{rem:parallel}
Let us fix a lattice polytope $\Delta\subseteq M\otimes\qq$
and let us consider $m := \lw(\Delta) + 1$. 
Observe that if $v\in N$ is a width direction,
then all the lattice points of 
$\Delta$ lie on $m$ parallel hyperplanes
of equation $v\cdot x - \alpha = 0$,
with $\alpha\in \{\min\langle \Delta,v\rangle,\dots,
\max\langle\Delta,v\rangle\}$. The product
of these hyperplanes gives a hypersurface
whose homogeneous part of higher degree
is $(v\cdot x)^m$, so that, by Theorem~\ref{bs},
$[v]$ is a base point for the $m$-th fundamental
form of $X(\Delta)$ at $\mathbf 1$ (if the form
is not empty), and hence it is also a base point
for $\pi^*H - mE$. Moreover, a similar argument
shows that any positive multiple of $\pi^*H - mE$ 
has the same point $[v]$ in its base locus,
so that the divisor $\pi^*H - mE$ is not semiample.

We remark that the above conclusion easily 
follows either from Proposition~\ref{prop:w}
or Proposition~\ref{prop:psnef},
since $\pi^*H - mE$ has negative intersection
with the curve $\tilde{C}_v$, so that  
it is not pseudonef (and hence it is not nef nor
semiample) and in particular the tangent direction
$[v]$ lies in the stable base locus of  $\pi^*H - mE$.

On the other hand, if we take $m := \lw(\Delta)$,
by Proposition~\ref{prop:psnef} we have that
$\pi^*H - mE$ is pseudonef, but we do not
know whether it is nef (or semiample). 
Our next goal is to give a sufficient condition
on $\Delta$, based on Theorem~\ref{bs},
which guarantees that $\pi^*H - mE$ is not semiample.
\end{remark}

\begin{corollary}
\label{cor:notsem}
Let $v\in N$ be a width direction
for $\Delta$, and suppose
that  the following conditions hold:
\begin{enumerate}
\item  the hyperplanes $L_{\min}\langle\Delta,v\rangle$
and $L_{\max}\langle\Delta,v\rangle$ 
intersect $\Delta$ exactly in two vertices,
$p_{\min}$ and $p_{\max}$ respectively;
\item the linear span $\Lambda$
of the lattice points on 
 $\{  \langle x,v\rangle = 
 \min\langle\Delta,v\rangle + 1\}\cap \Delta$
has codimension at least $2$;
\item the line through the two vertices
$p_{\min}$ and $p_{\max}$ does not
intersect $\Lambda$.
\end{enumerate}
Then $\pi^*H-\lw(\Delta)E$ is not semiample.
\end{corollary}
\begin{proof}
By translating $\Delta$ we can suppose that
$\Lambda$ passes through the origin and
(by taking $-v$ instead of $v$ if necessary)
that $\min\langle\Delta,v\rangle = - 1$ and 
$\max\langle\Delta,v\rangle = m-1$, where we set 
$m := \lw(\Delta)$ for simplicity of notation.

By (2) and (3) 
there exist two non-associated homogeneous
polynomials $f,g$ of degree one, 
such that $\{p_{\min}\}
\cup\Lambda\subseteq V(f)$ and 
$\{p_{\max}\}\cup\Lambda\subseteq V(g)$.
Moreover, by (3) and the fact that $v$ is
constant along $\Lambda$, we can 
choose $f$ and $g$ in such a way that
the hyperplane orthogonal to
$v$ and containing $\Lambda$
has equation $f+ g = 0$. Therefore the 
lattice points 
of $\Delta\cap\Lambda$ together 
with $p_{\min}$ and $p_{\max}$ lie on
the affine quadric of equation
\[
 (f+g)^2 - f(p_{\max}) f - g(p_{\min}) g = 0.
\]
We conclude that all the lattice points
of $\Delta$ lie on the union of
this quadric and the $m-2$
hyperplanes defined by 
$f+g -\alpha$, for 
$\alpha = 1, \dots, m-2$. 
This union is a degree $m$ affine
hypersurface whose
homogeneous part of maximal degree
is the power $(f+g)^m = (v\cdot x)^m$. 
We illustrate the above results
with the following two pictures.
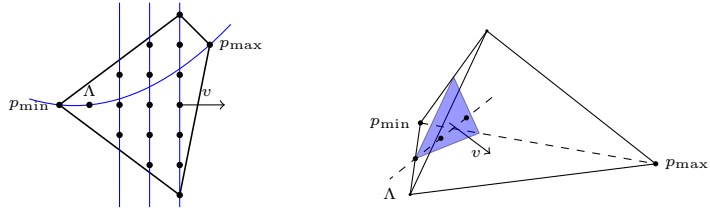
\begin{center}
 \begin{tikzpicture}[scale=.4]
  \tkzDefPoint(0,3){P}
 \tkzDefPoint(4,0){P1}
 \tkzDefPoint(5,5){P2}
 \tkzDefPoint(4,6){P3}
\tkzDefPoint(3,.75){A}
\tkzDefPoint(3,5.25){B}
\tkzDefPoint(2,1.5){C}
\tkzDefPoint(2,4.5){D}
 \tkzFillPolygon[color = black!0](P,P1,P2,P3)
 \tkzDrawSegments[color=black](P,P1 P1,P2 P2,P3 P3,P)
 
  \foreach \x in {2,3,4}
{\draw[blue] (\x,-.4) -- (\x,6.4);}

\draw[blue, domain=-1:5.7] plot (\x, {3-(\x)/10+(\x*\x)/10});
 \foreach \x/\y in {0/3,1/3,2/2,2/3,2/4,3/1,3/2,3/3,3/4,3/5,
 4/0, 4/1, 4/2, 4/3, 4/4, 4/5, 4/6, 5/5}
 {
  \tkzDefPoint(\x,\y){p}
  \tkzDrawPoints[fill=black,color=black,size=5](p)
 }

\draw [->] (4,3) -- (5.5,3);
\draw (4.9,3) node[above] {\tiny $v$};

\draw (P) node[left] {\tiny $p_{\min}$};
\draw (P2) node[right] {\tiny $p_{\max}$};
\draw (1,3) node[above] {\tiny $\Lambda$};
\begin{scope}[xshift = 12cm,yshift=2.4cm,scale=1.7,x={(-5mm,-4mm)},z={(0mm,10mm)},y={(4mm,-3mm)}]

\draw [->] (-1/3,1,1/6) -- (-1/3,3,1/6); 
\draw (0,2,0) node [right] {\tiny $v$};

\draw (0,0,0)--(2,2,0);
\draw (0,0,0)--(-1,2,2);
\draw[dashed] (0,0,0)--(-6,4,-2);
\draw (2,2,0)--(-1,2,2);
\draw (2,2,0)--(-6,4,-2);
\draw (-1,2,2)--(-6,4,-2);

\draw[fill=black] (0,0,0) circle (.4mm)
 node[left] {\tiny $p_{\min}$};
\draw[fill=black] (-6,4,-2) circle (.4mm)
 node[right] {\tiny $p_{\max}$};
\draw[fill=black] (2,2,0) circle (.2mm);
\draw[fill=black] (-1,2,2) circle (.2mm);

\draw[fill=blue,opacity=0.4] ( 1, 1,  0) -- ( -1/2, 1,  1) -- ( -3/2,  1, -1/2) -- cycle;

\coordinate (E) at (-1,1,0);
\coordinate (F) at (0,1,0);
\coordinate (G) at (1,1,0);

\draw[fill=black] (E) circle (.4mm);
\draw[fill=black] (F) circle (.4mm);
\draw[fill=black] (G) circle (.4mm);

\draw (2,1,0) node[below] {\tiny $\Lambda$};

\draw[dashed] (-2,1,0) -- (2,1,0);
\end{scope}
 \end{tikzpicture}
\end{center}

\vspace{4mm}

By Theorem~\ref{bs}, 
the $m$-th fundamental form of $X(\Delta)$
at $\mathbf 1$ has a base point corresponding
to $[v]\in\pp^{k-1}$,
and by Remark~\ref{rem:bs} we deduce that
$\pi^*H - mE$ has a base point too.

We now claim that the hypotheses are indeed
satisfied by any positive multiple $r\Delta$.
This is immediately clear for (1).
Concerning (2), observe that the 
cone $\rr_{>0}\cdot(\Delta - p_{\min})$ at 
$p_{\min}$ coincides with the cone 
$\rr_{>0}\cdot(r\Delta - rp_{\min})$ at $rp_{\min}$.
In particular the dimension of the linear 
span of lattice points in $\{  \langle x,v\rangle = 
 \min\langle r\Delta,v\rangle + 1\}\cap r\Delta$
 is equal to the dimension of $\Lambda$.
We illustrate this in the following 
picture (the corresponding toric variety is the  
fake projective plane obtained by
quotienting $\mathbb P^2$ by the
action of $\mathbb Z/27\mathbb Z$
defined by $\varepsilon\cdot [x_0:x_1:x_2]
= [x_0:\varepsilon^{20} x_1:\varepsilon x_2]$).

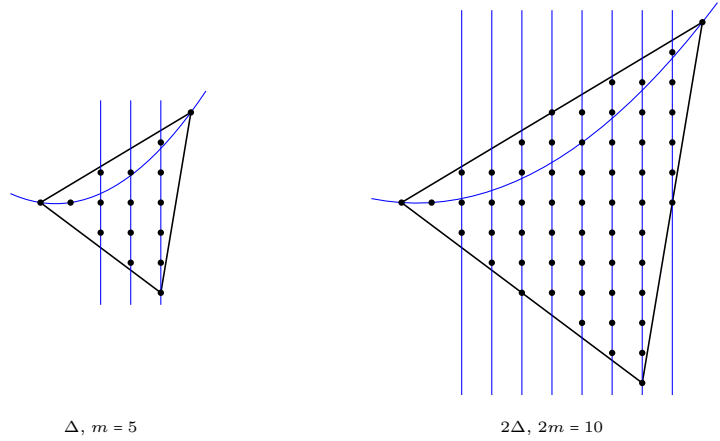
\vspace{4mm}

\begin{center}
 \begin{tikzpicture}[scale=.4]
 \begin{scope}

  \tkzDefPoint(0,0){P}
 \tkzDefPoint(4,-3){P1}
 \tkzDefPoint(5,3){P2}
 
 \tkzDefPoint(2,-1.5){P3}
 \tkzDefPoint(2,1.2){P4}
 
  \tkzDefPoint(3,-2.25){P5}
 \tkzDefPoint(3,1.8){P6}

\tkzDefPoint(4,2.4){P7}

 \tkzFillPolygon[color = black!0](P,P1,P2)
 \tkzDrawSegments[color=black](P,P1 P1,P2 P2,P)
 \draw[blue, domain=-1:5.5] plot (\x, {-3*(\x)/20+3*(\x*\x)/20});
 \foreach \x in {2,3,4}
{\draw[blue] (\x,-3.4) -- (\x,3.4);}

\foreach \y in {-1,0,1}
 { \fill[black] (2,\y) circle (3pt); }

\foreach \y in {-2,-1,0,1}
 { \fill[black] (3,\y) circle (3pt); }

\foreach \y in {-3,-2,-1,0,1,2}
 { \fill[black] (4,\y) circle (3pt); }

 \fill[black] (0,0) circle (3pt);
 \fill[black] (1,0) circle (3pt);
 \fill[black] (5,3) circle (3pt);

 \node[black] at (2,-7.5) 
 {\tiny $\Delta,\,m=5$};

 \end{scope}
 
 \begin{scope}[xshift=12cm]
 \tkzDefPoint(0,0){P}
 \tkzDefPoint(8,-6){P1}
 \tkzDefPoint(10,6){P2}
 
 \tkzDefPoint(2,-1.5){P3}
 \tkzDefPoint(2,1.2){P4}
 
  \tkzDefPoint(3,-2.25){P5}
 \tkzDefPoint(3,1.8){P6}

\tkzDefPoint(4,2.4){P7}

 \tkzFillPolygon[color = black!0](P,P1,P2)
 \tkzDrawSegments[color=black](P,P1 P1,P2 P2,P)
\foreach \x in {2,3,4,5,6,7,8,9}
{\draw[blue, thin] (\x,-6.4) -- (\x,6.4);}
  \draw[blue, domain=-1:10.5] plot (\x, {-2*(\x)/30+2*(\x*\x)/30});

\foreach \y in {-1,0,1}
 { \fill[black] (2,\y) circle (3pt); }

\foreach \y in {-2,-1,0,1}
 { \fill[black] (3,\y) circle (3pt); }

\foreach \y in {-3,-2,-1,0,1,2}
 { \fill[black] (4,\y) circle (3pt); }

\foreach \y in {-3,-2,-1,0,1,2,3}
 { \fill[black] (5,\y) circle (3pt); }

\foreach \y in {-4,-3,-2,-1,0,1,2,3}
 { \fill[black] (6,\y) circle (3pt); }
\foreach \y in {-5,-4,-3,-2,-1,0,1,2,3,4}
 { \fill[black] (7,\y) circle (3pt); }
\foreach \y in {-6,-5,-4,-3,-2,-1,0,1,2,3,4}
 { \fill[black] (8,\y) circle (3pt); }
\foreach \y in {0,1,2,3,4,5}
 { \fill[black] (9,\y) circle (3pt); }

 \fill[black] (0,0) circle (3pt);
 \fill[black] (1,0) circle (3pt);
 \fill[black] (10,6) circle (3pt);
 
 \node[black] at (5,-7.5) 
 {\tiny $2\Delta,\,2m=10$};

 \end{scope}

 \end{tikzpicture}
\end{center}
Finally, by the  hyperplane separation 
theorem,
the line through $p_{\min}$
and $p_{\max}$ is separated from $\Lambda$
by a hyperplane and this property is
preserved by dilations, which proves (3).
We conclude that any positive multiple of
$\pi^*H - mE$ has the same base point,
and the statement follows.
\end{proof}

\subsection{Toric surfaces}
\label{k=m=2}
Let us focus now on the case  
$k=2$, so that $\Delta\subseteq M\otimes\qq$ 
is a lattice polygon and $X(\Delta)$ is a projective 
toric surface. We first present some nice
consequences of the previous results and
of Proposition~\ref{pro:pol}, and then we
give a proof for the latter.

\begin{remark}
 \label{rem:spe}
Let us consider a lattice polygon $\Delta$
such that $|\Delta\cap M| \geq 6$, and the
corresponding toric pair $(X,H)$. 
If the linear system $|\pi^*H-3E|$
is special, by Proposition~\ref{pro:pol}
we have that $\lw(\Delta) \leq 2$, and in particular
Proposition~\ref{prop:w} implies 
that the base locus of $|\pi^*H-3E|$ contains a curve
intersecting $E$.

We are now going to give a counterexample
showing that the above result 
is no longer true for systems of the 
form $|\pi^*H-mE|$, when $m\geq 4$.
\end{remark}

\begin{example}
\label{ex:base}
Let $k\geq 2$ be an integer and
let us consider the polygon 
$\Delta\subseteq M\otimes_{\mathbb Z} \qq$ 
with vertices 
$(0,-k-1),(k,0),(-k,0),(0,k-1)$. 
The lattice width of $\Delta$ is $2k$, 
and there exist curves of degree $2k-1$, 
but not smaller, passing through all its 
lattice points (the following 
is the picture for $k=2$).
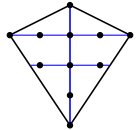
\vspace{4mm}
\begin{center}
\begin{tikzpicture}[scale=.4]
 \tkzDefPoint(0,0){P1}
 \tkzDefPoint(2,3){P2}
 \tkzDefPoint(0,4){P3}
 \tkzDefPoint(-2,3){P4}
 \tkzFillPolygon[color = black!0](P1,P2,P3,P4)
 \tkzDrawSegments[color=black](P1,P2 P2,P3 P3,P4 P4,P1)
 \tkzDrawSegments[color=blue](P1,P3)
  \draw[color=blue] (-1.3,2) -- (1.3,2);
  \draw[color=blue] (-2,3) -- (2,3);
 \foreach \x/\y in {0/1,0/2,0/3,0/4,1/2,-1/2,1/3,-1/3,0/0,2/3,-2/3}
 {
  \tkzDefPoint(\x,\y){p}
  \tkzDrawPoints[fill=black,color=black,size=5](p)
 }

\end{tikzpicture}
\end{center}

Therefore, if we consider the corresponding 
toric pair $(X,H)$, by Remark~\ref{rem:spec} 
we have that $h^1(\pi^*H-(2k)E) > 0$. 
We now claim that the base locus of $|\pi^*H-(2k)E|_E$ 
(equivalently, the base locus of the $2k$-th fundamental 
form of $X$ at $\mathbf 1$) is empty.
Indeed, if it were not the case,
by Theorem~\ref{bs} there would exist
a curve of degree $2k$ passing through the 
lattice points of $\Delta\cap M$ and 
having an inflection point of order $2k$
at infinity. But since there are $2k+1$ lattice
points of $\Delta$ on each of the two axes,
any curve of degree $2k$ passing through
$\Delta\cap M$ must contain the factor $xy$, 
so that it can not have an inflection of maximal 
order at infinity, which proves the claim. Observe
that on one hand this implies that the system 
$|\pi^*H-(2k)E|$ is not empty (and hence 
it is special), and on the other hand 
that the base locus of $|\pi^*H-(2k)E|$ 
does not contain any curve intersecting $E$
(when $k$ is small, it is possible
to prove, by a heavier calculation, that 
the above base locus is indeed empty).

We also remark that even if the toric surfaces 
corresponding to the above polytopes are singular, 
it is possible to resolve their singularities in order 
to obtain smooth examples.
\end{example}

\begin{remark}
There exist other examples
of toric surfaces such that $|\pi^*H-mE|$
is special but it contains no curve in its 
base locus, such as the Togliatti
surface (see Example~\ref{ex:pal}). But
the toric surfaces appearing in 
Example~\ref{ex:base} have the
additional property of being linearly normal,
since they correspond to all the lattice points
of a lattice polygon (see for 
instance~\cite{CLS}*{Chapter~2}).
\end{remark}

\begin{remark}
 \label{rem:teo}
Another direct consequence of Proposition~\ref{pro:pol}
is that, when $k=m=2$ the first three conditions 
of Proposition~\ref{prop:w} are equivalent.
Moreover, if $X$ is smooth, they are also
equivalent to the fourth one, $h^1(\pi^*H-3E)  > 0$.
Indeed, the homomorphism 
$\rho_1$ defined in~\eqref{rho} is
surjective since the hyperplanes of $\mathbb P^n$ 
through $\mathbf 1$ do not have a fixed direction.
Hence $h^1(\pi^*H-2E) = 0$, so that
the hypothesis $h^1(\pi^*H-3E) > 0$
implies that the map $\rho_2$ is not surjective.
In particular the second fundamental form 
is not full dimensional. 
By Proposition~\ref{pro:pol} and the fact
that $\Delta$ is smooth, the latter must 
be a Cayley polytope.

If $m=2$ and $k \geq 3$, condition $(3)$ of 
Proposition~\ref{prop:w} implies (by Theorem~\ref{bs}) that
the lattice points of $\Delta$ lie on an affine 
paraboloid $V(l_1^2-l_2)$, with $l_1$ and 
$l_2$ linear forms. In particular 
$\Delta$ can not have any lattice 
point in its interior by the convexity
of the paraboloid, i.e. $\Delta$ is a so called 
{\em hollow polytope}.
For $k=3$ there exists a complete classification 
of hollow polytopes (see~\cites{AKW,Ra}),
and looking at the list we can again conclude 
that, under our hypotheses, it must be $\lw(\Delta) = 1$,
so that $(3) \Rightarrow  (1)$ (and hence the first
three conditions are equivalent).
For $k > 3$, we believe that the implication
still holds true, but since in this case there is no
complete classification of hollow polytopes, 
so far we were not able to prove the result.

Finally, for bigger values of $m$, the 
implication $(3) \Rightarrow  (1)$
of Proposition~\ref{prop:w}
is no longer true in general
(even in dimension $k=2$).
For instance, given any polygon 
$\Delta$ satisfying the hypotheses 
of Corollary~\ref{cor:notsem},
we have seen that the $m$-th fundamental 
form of $X(\Delta)$ has a base point,
but $\lw(\Delta)=m$. 

\end{remark}

\begin{proof}[Proof of Proposition~\ref{pro:pol}]
The equivalence of $(1)$ and $(2)$ 
follows from Proposition~\ref{equiv}.
Let us prove $(2) \Rightarrow (3)$. 
By Remark~\ref{rem:spec}, the lattice points of $\Delta$
lie on a conic, which by Lemma~\ref{3online}
below is the union of two lines, 
say $L_1$ and $L_2$. 
If they are parallel, they must be at lattice
distance $1$, so that $\Delta$ is a Cayley
polygon.
Let us consider then the case in which
$L_1$ and $L_2$ meet in a point $q$.
We set  $r_i:=|\Delta\cap L_i\cap M|$, for $i=1,2$,
and we suppose that $r_1\geq r_2$ (so that
in particular $r_1\geq 3$). 
We also assume that $L_1$ coincides with the $x$-axis
and we distinguish the following two cases.

a) $r_1\geq 4.$ 
We claim that in this case any point $p=(x_p,y_p)\in 
\Delta\cap(L_2\setminus L_1)$ 
satisfies $|y_p| \leq 1$. Let us suppose on the contrary
that $|y_p| \geq 2$ and let us consider the intersection
of the line $y=\pm1$ (depending on the sign of $y_p$)
with the triangle generated by $p$ and $\Delta\cap L_1$. 
If $|y_p| \geq 3$ or $r_1\geq 5$, the length of this segment 
is at least $2$, while if $|y_p| = 2$ and $r_1=4$, 
the length is $3/2$, but one of its endpoints is a lattice 
point, so that in both cases we 
have at least two lattice points on this segment.
This is a contradiction and proves the claim. In particular
there are at most $2$ lattice points on
$\Delta\cap(L_2\setminus L_1)$. If there is only
one point, $\Delta$ turns out to be a Cayley triangle.
If there are $2$ points, $\Delta$ is equivalent to 
a polygon of type $(i)$ or $(ii)$, depending on 
whether $q$ is a lattice point or not.

b) $r_1 = 3.$ In this case we must have
$r_2 = 3$ and the intersection point $q$ 
is not a lattice point. We are going to show that 
it not possible. First of all observe that on one of
the two half-lines determined by $q$ on $L_1$
(resp. $L_2$) there are at least $2$ lattice points, 
say $(a,0)$ and $(a+1,0)$ (resp. $p_1$ and $p_2$). 
Moreover the triangles with vertices $(a,0),\, (a+1,0)$ and 
$p_i$ for $i=1,2$, contain no lattice point but
the vertices, which implies that they both have area $1/2$.
Therefore $p_1$ and $p_2$ lie on one
of the two lines $y=\pm 1$,
contradicting the fact that $L_1$ and $L_2$ 
intersect. This concludes the proof of $(2)\Rightarrow (3)$.

In order to prove $(3) \Rightarrow (1)$
observe that if $\Delta$ is either Cayley
or of type $(i)$ or $(ii)$, its width 
satisfies the inequality $\lw(\Delta) \leq 2$ 
and hence we conclude by means of 
Proposition~\ref{prop:w}.

Finally, concerning the last assertion,
by Theorem~\ref{bs} the second fundamental 
form at $\mathbf 1$ has a base point if and 
only if the lattice points of $\Delta\cap M$ lie 
on a parabola. 
Since in this case the parabola is degenerate,
it must be the union of two parallel lines at lattice
distance $1$, which is equivalent to say that
$\Delta$ is Cayley.
\end{proof}

We conclude this subsection with the
the following lemma that we used in the proof
of Proposition~\ref{pro:pol} (we believe that its
content is well known, but we give a proof anyway
since we could not find any explicit reference to this result).
\begin{lemma}
\label{3online}
Let $\Delta\subseteq M\otimes_{\mathbb Z}\qq$ be a lattice polygon
such that $|\Delta\cap M| \geq 5$.
Then there are at least $3$ lattice points
of $\Delta$ lying on a line.
\end{lemma}
\begin{proof}
We can reduce to the case $|\Delta\cap M| = 5$.
In this case, removing one of the vertices and
taking the convex hull of the remaining $4$ lattice
points we obtain a lattice polygon $\Delta'$ with
$4$ lattice points and we claim that it is equivalent
to one of the following.
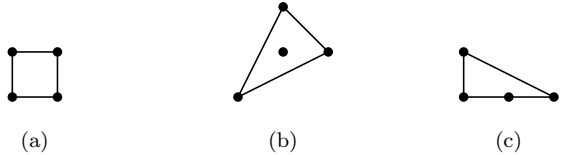
\vspace{3mm}

\begin{center}
 \begin{tikzpicture}[scale=.6]
 \begin{scope}

 \tkzDefPoint(0,0){P}
 \tkzDefPoint(1,0){P1}
 \tkzDefPoint(1,1){P2}
 \tkzDefPoint(0,1){P3}

 \tkzFillPolygon[color = black!0](P,P1,P2,P3)
 \tkzDrawSegments[color=black](P,P1 P1,P2 P2,P3 P3,P)

\foreach \y in {0,1}
 { \fill[black] (0,\y) circle (3pt); }

\foreach \y in {0,1}
 { \fill[black] (1,\y) circle (3pt); }

 \node[black] at (.5,-1) 
 {\footnotesize (a)};

 \end{scope}
 
 \begin{scope}[xshift=6cm]

 \tkzDefPoint(0,1){P}
 \tkzDefPoint(1,1){P1}
 \tkzDefPoint(-1,0){P2}
 \tkzDefPoint(0,2){P3}

 \tkzFillPolygon[color = black!0](P1,P2,P3)
 \tkzDrawSegments[color=black](P1,P2 P2,P3 P3,P1)

\foreach \y in {1,2}
 { \fill[black] (0,\y) circle (3pt); }
 \fill[black] (1,1) circle (3pt); 
 \fill[black] (-1,0) circle (3pt); 
 
  \node[black] at (0,-1) 
 {\footnotesize (b)};

 \end{scope}

 \begin{scope}[xshift=10cm]

  \tkzDefPoint(0,0){P}
 \tkzDefPoint(1,0){P1}
 \tkzDefPoint(2,0){P2}
 \tkzDefPoint(0,1){P3}

 \tkzFillPolygon[color = black!0](P,P2,P3)
 \tkzDrawSegments[color=black](P,P2 P2,P3 P3,P)

\foreach \x in {0,1,2}
 { \fill[black] (\x,0) circle (3pt); }

\fill[black] (0,1) circle (3pt);
  \node[black] at (1,-1) 
 {\footnotesize (c)};

\end{scope}
 \end{tikzpicture}
\end{center}
Indeed, if $\Delta'$ is a quadrilateral then by
Pick's Theorem its area is $1$, so that it is equivalent
to the unitary square. If $\Delta'$ is a triangle with
one internal lattice point, joining this point
with any pair of vertices we obtain a triangle 
of area $1/2$. Then $\Delta'$ is equivalent to 
the triangle (b). Finally, if $\Delta'$ is a triangle
without internal lattice points, its area is $1$.
Acting with $\GL(2,\zz)$ we can suppose that
the edge having one lattice point in its relative 
interior lies on the $x$-axis, which means that the 
height of the triangle is $1$, so that it is
equivalent to (c). This proves the claim.

We conclude by observing that 
if we add one lattice point to any
of the polygons above and we take the
convex hull, we obtain at least three
lattice points on a line.
\end{proof}

\subsection{Weighted projective spaces}
In this last subsection we are going to prove
Proposition~\ref{propo:nonmds}. From now on
we restrict our study to weighted projective 
spaces $\pp(a_1,\dots,a_{k+1})$,
where $w=[a_1,\dots,a_{k+1}] \in (\zz_{>0})^{k+1}$
is the vector of weights. 
Let us denote by $A$ a generator 
for the divisor class group of $X$ (which is
torsion-free of rank one). In what follows we recall
how to construct a (non unique) lattice simplex 
$\Delta \subseteq M\otimes_{\zz}\qq$, such that
the toric pair $(X,H)$ associated to $\Delta$
consists of $X = \pp(a_1,\dots,a_{k+1})$ and 
$H := dA$, where $d:=\lcm(w)$.
Let us consider the following exact sequence
\[
 \xymatrix{
0 \ar[r] & M \ar[r]^{P^*} & \tilde M \ar[r]^{\cdot w} & \zz \ar[r] & 0,
}
\]
where $P^*$ is the kernel matrix of the product
by $w$, so that $M$ is included in $\tilde M\simeq \zz^{k+1}$ 
as the orthogonal to $w$. We can define the simplex
\[
 \tilde\Delta
 := 
 \{u\in \tilde M\otimes_\zz\qq
 \, :\,
 u\cdot w = \lcm(w)\, {\rm and }\, u_i\geq 0\},
\]
i.e. the intersection of the 
non negative ortant of $\tilde M$ with the
affine hyperplane defined by $u\cdot w = d$.
Observe that if we denote by $e_1,\dots,e_{n+1}$ 
the canonical basis of $\tilde M$,
then the vertices of $\tilde\Delta$ are $\delta_ie_i$,
where $\delta_i := d/a_i$. If we now translate 
$\tilde\Delta$ moving one of its vertices (for instance 
the first one) to the origin
and we pull it back to $M$ via $P^*$, we obtain
the simplex
\[
 \Delta := (P^*)^{-1}(\text{chull}(\delta_ie_i-\delta_1e_1\, :\, i\in\{1,\dots,n+1\})).
\]
We can now prove the following sufficient
condition for the nefness of the divisor $\pi^*dA-\lw(\Delta)E$
on the blowing up of $\pp(w)$ at the point $\mathbf 1$.
\begin{proposition}
\label{nef:p}
Let $v\in N$ be a width direction
such that the corresponding one 
parameter subgroup is the
intersection of hypersurfaces of degree 
smaller than $d/\lw(\Delta)$.
Then $\pi^*dA-\lw(\Delta)E$ is nef
or equivalently
$\Nef(\tilde X) = \PNef(\tilde X)$.
\end{proposition}
\begin{proof}
For simplicity of notation let us set 
$\tilde D = \pi^*dA-\lw(\Delta)E$,
and let us consider the curve 
$C_v\subseteq X$ (see Notation~\ref{def:param}).
By Remark~\ref{rem:deg}
we have that $\tilde{C}_v\cdot \tilde D = 0$,
where $\tilde{C}_v$ is the strict transform 
of $C_v$.
Moreover, since the divisor class group of 
$\tilde X$ has rank two, the nef cone and 
its dual, the Mori cone, are two dimensional,
so that in order to prove that $\tilde D$ is nef
it is enough to show that $\tilde{C}_v$ generates 
an extremal ray of the Mori cone. 

Let $n$ be a positive integer
such that $n\tilde{C}_v\equiv \tilde C_1 + \tilde C_2$,
with $\tilde C_1$ and $\tilde C_2$ 
effective curves of $\tilde X$.
Let $C_v = D_1\cap\dots\cap D_{r}$ and
let $\tilde D_j\equiv\alpha_j\pi^* A-E$ be 
the strict transform of $D_j$. 
By hypothesis $\alpha_j < d/\lw(\Delta)$
for any $1\leq j \leq r$.
We claim that $\tilde C_i\cdot \tilde D\geq 0$, for 
$i =1$ and $2$. Let us suppose by contradiction that
$\tilde C_1\cdot \tilde D < 0$.
Then at least one irreducible component 
$\tilde \Gamma$ of $\tilde C_1$ would intersect
negatively $\tilde D$ and thus
\[
 \tilde \Gamma\cdot \tilde D_j 
 =
 \tilde\Gamma\cdot(\alpha_j\pi^*A-E)
 <
 \tilde\Gamma\cdot(d/\lw(\Delta)\pi^*A-E)
 = 
 \frac{1}{\lw(\Delta)}\tilde \Gamma\cdot \tilde D < 0
\]
for any $j= 1,\dots,r$.
In particular $\tilde \Gamma$ would be 
contained in the intersection of all the 
$\tilde D_j$, so that $\tilde \Gamma
= \tilde C_v$, because $\tilde C_v$ is
irreducible. But this contradicts the equality
$\tilde C_v\cdot \tilde D = 0$ and proves the claim.
Using the equalities 
\[
 0 
 = n\tilde C_v\cdot \tilde D
 = (\tilde C_1 + \tilde C_2)\cdot \tilde D
\]
and the claim, we conclude that 
$\tilde C_i\cdot \tilde D = 0$
for each $i=1,2$, which implies that
the classes of $\tilde C_1$ and $\tilde C_2$ 
are proportional to that of $\tilde C_v$.
Therefore $\tilde D = \pi^*dA-\lw(\Delta)E$ 
is nef. Moreover, since the Picard group of 
$\tilde X$ has rank $2$, by Proposition~\ref{prop:psnef} 
we conclude that $\Nef(\tilde X) =\PNef(\tilde X)$. 
\end{proof}

\begin{proof}[Proof of Proposition~\ref{propo:nonmds}]
The proof is a case by case analysis
performed with the Computer Algebra
package Magma~\cite{mag} and
we explain it in the next lines.
Given any $w = [a_1,\dots,a_4]$ appearing in the
list, we construct $\Delta$ as explained at the beginning
of the section. Let $m := \lw(\Delta)$ be the
lattice width of $\Delta$ and let $v\in N$
be a width direction for $\Delta$.
In each case we verify that the hypotheses of 
Corollary~\ref{cor:notsem} 
are satisfied for the pair
$(\Delta, v)$, so that $\pi^*dA-mE$ is not 
semiample.

In order to conclude the proof
it is enough to prove that $\pi^*dA-mE$
is nef by means of 
Proposition~\ref{nef:p}. 
Consider the kernel $I$ of the map
\begin{equation}
 \label{eq:I}
 \mathbb C[x_1,\dots,x_4]\to\mathbb C[t],
 \quad x_i\mapsto t^{a_i},\text{ for $i=1\dots,4$},
\end{equation}
which is an ideal generated by binomials
and it is called {\em lattice ideal}, 
according to~\cite{hkl}*{\S 5}.
Let $(f_1,\dots,f_r)$ be a minimal basis
for $I$, ordered by increasing $w$-degree,
and let $d_i$ be the $w$-degree of $f_j$, 
for any $j = 1,\dots, r$.
We verified that in all but 
one case the following holds
\[
 \deg V(f_1,f_2)
  = m,
\]
where the degree is calculated with 
respect to the very ample class $dA$,
i.e. it is given by $d_1A\cdot d_2A\cdot dA 
= d_1d_2d/(a_1a_2a_3a_4)$.
This immediately implies that the
complete intersection $V(f_1,f_2)$
is irreducible, since otherwise one of its
prime components 
would be a one-parameter subgroup
of degree smaller than the width
$m$ of $\Delta$, a contradiction. 
In all these cases one verifies that 
$d_i < d/\lw(\Delta)
= {\rm lcm}(w)/m$, for $i=1,2$, so that 
$\pi^*dA-mE$ is nef by 
Proposition~\ref{nef:p}.
Finally in the remaining case,
namely $[23,27,29,30]$,
the curve $V(f_1,f_2,f_3)$ 
is irreducible, where 
$f_1 = x_1x_3^2 - x_2^3,\,
f_2 = x_1x_4^2 - x_2^2x_3$
and $f_3 = x_2x_4^2 - x_3^3$,
so that $d_1 = 81,\, d_2 = 83$ and 
$d_3 = 87$.
Also in this case we have that
$d_i < d/\lw(\Delta)
= {\rm lcm}(w)/m = 690/7$, for any 
$i=1,2,3$, and we conclude
again by means of Proposition~\ref{nef:p}.
\end{proof}

\begin{remark}
 \label{rem:can}
 Even if in all the cases studied in the preceding proof
 the width of $\Delta$ is realised by the binomials of lowest degree
 of the lattice ideal $I$, in general this is not true.
 For example let us consider the vector of weights $w = [4,5,6,7]$. 
 In this case $d := \lcm(w) = 420$ and
 the inclusion of $M$ in $\tilde M$ is given by
 \[
 P^* =
 \begin{pmatrix}
 1 & 0 & 4 & -4\\
 0 & 1 & 5 & -5\\
 0 & 0 & 7 & -6
 \end{pmatrix}.
 \]
 Therefore the vertices of $\Delta$ are $(0,0,0),\,(-105,84,0),\,
 (-105,0,70)$ and $(-105,0,60)$, and $\lw(\Delta) = 
 \lw_v(\Delta) = 42$, where $v = (2,2,3)$.
  A minimal basis for the lattice ideal is 
  $(f_1,\dots,f_6) = (x_1x_3-x_2^2,\, 
  x_1x_4 - x_2x_3,\, x_2x_4-x_3^2,\ x_1^3-x_3^2,\,
  x_1^2x_2-x_3x_4,\, x_1x_2^2 - x_4^2)$.
  The dual of the inclusion $M\to \tilde M$ is 
  the map $\tilde N \to N$ given by the transpose of $P^*$.
  A preimage of $v$ in $\tilde N$ is the vector 
  $\tilde v = (2,2,3,3)$, defining a one parameter
  subgroup whose closure is a prime component of
  $V(f_2,f_4)$, but not of $V(f_1,f_2)$.  
\end{remark}

\begin{remark}
\label{rem:ex}
When $k=2$, the second condition of
Corollary~\ref{cor:notsem} states
that there is only one lattice point on the 
line at lattice distance $1$ from one of the
two vertices $p_{min}$ and $p_{max}$
of the triangle.
Therefore Corollary~\ref{cor:notsem} 
turns out to be equivalent 
to~\cite{GKK}*{Theorem~1.5},
in case $n=1$ (with the notation of~\cite{GKK}), 
so that we could not find any new
example in the class of weighted projective 
planes. 
Anyway, our technique is different from
the one used by the authors of the cited
paper. Their generalisation
to dimension $3$ (see~\cite{GKK1}*{Theorem~2.11})
gives rise to the list appearing 
in~\cite{GKK1}*{Table~1}, where there is only one example
with $a_i\leq 30$ for any $i$, namely $\pp(17,18,20,27)$.

\end{remark}


\begin{bibdiv}
\begin{biblist}

\bib{AH}{article}{
   author={Alexander, J.},
   author={Hirschowitz, A.},
   title={Polynomial interpolation in several variables},
   journal={J. Algebraic Geom.},
   volume={4},
   date={1995},
   number={2},
   pages={201--222},
   issn={1056-3911},
   review={\MR{1311347}},
}



\bib{AKW}{article}{
   author={Averkov, Gennadiy},
   author={Kr\"{u}mpelmann, Jan},
   author={Weltge, Stefan},
   title={Notions of maximality for integral lattice-free polyhedra: the
   case of dimension three},
   journal={Math. Oper. Res.},
   volume={42},
   date={2017},
   number={4},
   pages={1035--1062},
   issn={0364-765X},
   review={\MR{3722425}},
}

\bib{AWW}{article}{
   author={Averkov, Gennadiy},
   author={Wagner, Christian},
   author={Weismantel, Robert},
   title={Maximal lattice-free polyhedra: finiteness and an explicit
   description in dimension three},
   journal={Math. Oper. Res.},
   volume={36},
   date={2011},
   number={4},
   pages={721--742},
   issn={0364-765X},
   review={\MR{2855866}},
   doi={10.1287/moor.1110.0510},
}

\bib{mag}{article}{
    AUTHOR = {Bosma, Wieb and Cannon, John and Playoust, Catherine},
     TITLE = {The {M}agma algebra system. {I}. {T}he user language},
      NOTE = {Computational algebra and number theory (London, 1993)},
   JOURNAL = {J. Symbolic Comput.},
  FJOURNAL = {Journal of Symbolic Computation},
    VOLUME = {24},
      YEAR = {1997},
    NUMBER = {3-4},
     PAGES = {235--265},
      ISSN = {0747-7171},
   MRCLASS = {68Q40},
       URL = {http://dx.doi.org/10.1006/jsco.1996.0125},
}

\bib{Ca}{article}{
   author={Castravet, Ana-Maria},
   title={Mori dream spaces and blow-ups},
   conference={
      title={Algebraic geometry: Salt Lake City 2015},
   },
   book={
      series={Proc. Sympos. Pure Math.},
      volume={97},
      publisher={Amer. Math. Soc., Providence, RI},
   },
   date={2018},
   pages={143--167},
   review={\MR{3821148}},
}

\bib{CGG}{article}{
   author={Catalisano, M. V.},
   author={Geramita, A. V.},
   author={Gimigliano, A.},
   title={Higher secant varieties of Segre-Veronese varieties},
   conference={
      title={Projective varieties with unexpected properties},
   },
   book={
      publisher={Walter de Gruyter, Berlin},
   },
   date={2005},
   pages={81--107},
   review={\MR{2202248}},
}

\bib{CC}{article}{
   author={Chiantini, L.},
   author={Ciliberto, C.},
   title={Weakly defective varieties},
   journal={Trans. Amer. Math. Soc.},
   volume={354},
   date={2002},
   number={1},
   pages={151--178},
   issn={0002-9947},
   review={\MR{1859030}},
}

\bib{CLS}{book}{
   author={Cox, David A.},
   author={Little, John B.},
   author={Schenck, Henry K.},
   title={Toric varieties},
   series={Graduate Studies in Mathematics},
   volume={124},
   publisher={American Mathematical Society, Providence, RI},
   date={2011},
   pages={xxiv+841},
   isbn={978-0-8218-4819-7},
   review={\MR{2810322}},
}


\bib{DP}{article}{
   author={Dickenstein, Alicia},
   author={Piene, Ragni},
   title={Higher order selfdual toric varieties},
   journal={Ann. Mat. Pura Appl. (4)},
   volume={196},
   date={2017},
   number={5},
   pages={1759--1777},
   issn={0373-3114},
   review={\MR{3694743}},
}


\bib{Fu}{book}{
   author={Fulton, William},
   title={Introduction to toric varieties},
   series={Annals of Mathematics Studies},
   volume={131},
   note={The William H. Roever Lectures in Geometry},
   publisher={Princeton University Press, Princeton, NJ},
   date={1993},
   pages={xii+157},
   isbn={0-691-00049-2},
   review={\MR{1234037}},
 }
 
 \bib{fu2}{book}{
   author={Fulton, William},
   title={Intersection theory},
   series={Ergebnisse der Mathematik und ihrer Grenzgebiete. 3. Folge. A
   Series of Modern Surveys in Mathematics [Results in Mathematics and
   Related Areas. 3rd Series. A Series of Modern Surveys in Mathematics]},
   volume={2},
   edition={2},
   publisher={Springer-Verlag, Berlin},
   date={1998},
   pages={xiv+470},
   isbn={3-540-62046-X},
   isbn={0-387-98549-2},
   review={\MR{1644323}},
   doi={10.1007/978-1-4612-1700-8},
}

\bib{G}{book}{
   author={Gimigliano, Alessandro},
   title={On linear systems of plane curves},
   note={Thesis (Ph.D.)--Queen's University (Canada)},
   publisher={ProQuest LLC, Ann Arbor, MI},
   date={1987},
   pages={(no paging)},
   isbn={978-0315-38458-3},
   review={\MR{2635606}},
}

\bib{GKK}{article}{
    AUTHOR = {Gonz\'{a}lez, Jos\'{e} Luis},
    AUTHOR = {Karu, Kalle},
     TITLE = {Some non-finitely generated {C}ox rings},
   JOURNAL = {Compos. Math.},
    VOLUME = {152},
      YEAR = {2016},
    NUMBER = {5},
     PAGES = {984--996},
      ISSN = {0010-437X},
  }

\bib{GKK1}{article}{
    AUTHOR = {Gonz\'{a}lez, Jos\'{e} Luis},
    AUTHOR = {Karu, Kalle},
   title={Examples of non-finitely generated Cox rings},
   date={2017},
    JOURNAL = {arXiv:1708.09064},
    EPRINT =   {https://arxiv.org/pdf/1708.09064.pdf},
}



\bib{GH}{article}{
   author={Griffiths, Phillip},
   author={Harris, Joseph},
   title={Algebraic geometry and local differential geometry},
   note={},
   publisher={},
   date={1979},
   pages={355 -- 452},
 }

\bib{H}{article}{
   author={Harbourne, Brian},
   title={Free resolutions of fat point ideals on ${\mathbb P}^2$},
   journal={J. Pure Appl. Algebra},
   volume={125},
   date={1998},
   number={1-3},
   pages={213--234},
   issn={0022-4049},
   review={\MR{1600024}},
   doi={10.1016/S0022-4049(96)00126-0},
}

\bib{hkl}{article}{
   author={Hausen, J\"{u}rgen},
   author={Keicher, Simon},
   author={Laface, Antonio},
   title={Computing Cox rings},
   journal={Math. Comp.},
   volume={85},
   date={2016},
   number={297},
   pages={467--502},
}

\bib{He}{article}{
   author={He, Zhuang},
   title={Mori dream spaces and blow-ups of weighted
projective spaces},
 JOURNAL =  {arXiv:1803.11536},
  EPRINT =   {https://arxiv.org/pdf/1708.09064.pdf},
}		

\bib{Hi}{article}{
   author={Hirschowitz, Andr\'{e}},
   title={Une conjecture pour la cohomologie des diviseurs sur les surfaces
   rationnelles g\'{e}n\'{e}riques},
   language={French},
   journal={J. Reine Angew. Math.},
   volume={397},
   date={1989},
   pages={208--213},
   issn={0075-4102},
   review={\MR{993223}},
   doi={10.1515/crll.1989.397.208},
}

\bib{IR}{article}{
   author={Ionescu, Paltin},
   author={Russo, Francesco},
   title={Manifolds covered by lines and the Hartshorne conjecture for
   quadratic manifolds},
   journal={Amer. J. Math.},
   volume={135},
   date={2013},
   number={2},
   pages={349--360},
   issn={0002-9327},
   review={\MR{3038714}},
}

		
\bib{LU}{article}{
   author={Laface, Antonio},
   author={Ugaglia, Luca},
   title={On a class of special linear systems of $\mathbb P^3$},
   journal={Trans. Amer. Math. Soc.},
   volume={358},
   date={2006},
   number={12},
   pages={5485--5500},
   issn={0002-9947},
   review={\MR{2238923}},
}

\bib{LU1}{article}{
   author={Laface, Antonio},
   author={Ugaglia, Luca},
   title={Standard classes on the blow-up of $\mathbb P^n$ at points in very
   general position},
   journal={Comm. Algebra},
   volume={40},
   date={2012},
   number={6},
   pages={2115--2129},
   issn={0092-7872},
   review={\MR{2945702}},
 }

\bib{laz}{book}{
   author={Lazarsfeld, Robert},
   title={Positivity in algebraic geometry. I},
   series={Ergebnisse der Mathematik und ihrer Grenzgebiete. 3. Folge. A
   Series of Modern Surveys in Mathematics [Results in Mathematics and
   Related Areas. 3rd Series. A Series of Modern Surveys in Mathematics]},
   volume={48},
   note={Classical setting: line bundles and linear series},
   publisher={Springer-Verlag, Berlin},
   date={2004},
   pages={xviii+387},
   isbn={3-540-22533-1},
   review={\MR{2095471}},
   doi={10.1007/978-3-642-18808-4},
}
	
\bib{LS}{article}{
   author={Lubbes, Niels},
   author={Schicho, Josef},
   title={Lattice polygons and families of curves on rational surfaces},
   journal={J. Algebraic Combin.},
   volume={34},
   date={2011},
   number={2},
   pages={213--236},
   issn={0925-9899},
   review={\MR{2811147}},
   doi={10.1007/s10801-010-0268-y},
}

\bib{Pe}{article}{
   author={Perkinson, David},
   title={Inflections of toric varieties},
   note={Dedicated to William Fulton on the occasion of his 60th birthday},
   journal={Michigan Math. J.},
   volume={48},
   date={2000},
   pages={483--515},
   issn={0026-2285},
   review={\MR{1786502}},
}		

\bib{Ra}{article}{
   author={Rabinowitz, Stanley},
   title={A census of convex lattice polygons with at most one interior
   lattice point},
   journal={Ars Combin.},
   volume={28},
   date={1989},
   pages={83--96},
   issn={0381-7032},
   review={\MR{1039134}},
}

\bib{Se}{article}{
   author={Segre, Beniamino},
   title={Alcune questioni su insiemi finiti di punti in geometria
   algebrica. },
   language={Italian},
   conference={
      title={Atti Convegno Internaz. Geometria Algebrica},
      address={Torino},
      date={1961},
   },
   book={
      publisher={Rattero, Turin},
   },
   date={1962},
   pages={15--33},
   review={\MR{0146714}},
}

\bib{St}{book}{
   author={Sturmfels, Bernd},
   title={Gr\"obner bases and convex polytopes},
   series={University Lecture Series},
   volume={8},
   publisher={American Mathematical Society, Providence, RI},
   date={1996},
   pages={xii+162},
   isbn={0-8218-0487-1},
   review={\MR{1363949}},
}


\bib{To} {article}{
    AUTHOR = {Togliatti, Eugenio},
     TITLE = {Alcune osservazioni sulle superficie razionali che
              rappresentano equazioni di {L}aplace},
   JOURNAL = {Ann. Mat. Pura Appl. (4)},
    VOLUME = {25},
      YEAR = {1946},
     PAGES = {325--339},
      ISSN = {0003-4622},
}
		
\end{biblist}
\end{bibdiv}
\end{document}